\documentclass[12pt]{article}

\usepackage{amssymb}
\usepackage{graphicx}

\usepackage[subfigure]{tocloft}
\usepackage{color}
\usepackage{fancyvrb}
\usepackage{latexsym}
\usepackage{amssymb,amsmath,amstext}
\usepackage{epsfig}
\usepackage{graphics}
\usepackage{subfigure}
\usepackage{euscript}
\usepackage{ifthen}
\usepackage{wrapfig}
\usepackage{amsthm}
\usepackage{multicol}
\usepackage{multirow}
\usepackage{paralist}
\usepackage[margin=1in]{geometry}

\usepackage{multicol}
\usepackage[small,nohug,heads=vee]{diagrams}
\diagramstyle[labelstyle=\scriptstyle]

\usepackage[all]{xy}
\usepackage[colorlinks=true, urlcolor=blue, linkcolor=blue, citecolor=blue]{hyperref}
\usepackage[all]{hypcap}

\numberwithin{equation}{section}
\theoremstyle{plain}% 
\newtheorem{theorem}{Theorem}
\numberwithin{theorem}{section}
\newtheorem{proposition}[theorem]{Proposition}
\newtheorem{example}[theorem]{Example}
\newtheorem{lemma}[theorem]{Lemma}
\newtheorem{corollary}[theorem]{Corollary}

\newtheorem{remark}[theorem]{Remark}
\newtheorem{conjecture}[theorem]{Conjecture}

\newcommand{\C}{\mathbb{C}}
\newcommand{\Z}{\mathbb{Z}}
\newcommand{\PP}{\mathbb{P}}

\newcommand{\R}{\mathbb{R}}
\newcommand{\arxiv}[1]{\href{http://arxiv.org/abs/#1}{{\tt arXiv:#1}}}

\allowdisplaybreaks

\date{November 5, 2013}
\begin{document}

\title{\bf Tropicalization of classical moduli spaces}

\author{Qingchun Ren \and Steven V Sam \and Bernd Sturmfels}

\maketitle

  \begin{abstract}
  \noindent 
The image of the complement of a hyperplane arrangement
under a monomial map can be tropicalized combinatorially using 
matroid theory.  
We apply this to classical moduli spaces that are associated
with  complex reflection arrangements.
Starting from modular curves, we visit the
Segre cubic, the Igusa quartic, and
moduli of marked del Pezzo surfaces of degrees $2$ and $3$.  Our primary example is the
 Burkhardt quartic, whose tropicalization is
 a $3$-dimensional fan 
in  $39$-dimensional space.
This effectuates a synthesis of
concrete and abstract approaches to
tropical moduli of  genus 2 curves.
    \end{abstract}

%\tableofcontents

\section{Introduction}

Algebraic geometry is the study of solutions sets
to polynomial equations. Solutions that depend on
an infinitesimal parameter can be analyzed combinatorially using
min-plus algebra. This insight led to the development of 
tropical algebraic geometry \cite{MS}. 
While all algebraic varieties and their tropicalizations may be explored
at various level of granularity, varieties that serve as moduli spaces
are usually studied at the highest level of abstraction.
This paper does  exactly the opposite: we investigate and
tropicalize certain concrete moduli spaces,
mostly  from the 19th century repertoire \cite{hunt}, by means 
of their defining polynomials.

A first example, familiar to all algebraic geometers, is the
moduli space $\mathcal{M}_{0,n}$ of $n$ distinct points on the projective line $\mathbb{P}^1$.
We here regard $\mathcal{M}_{0,n}$ as a subvariety in a suitable torus.
Its tropicalization ${\rm trop}(\mathcal{M}_{0,n})$
is a simplicial fan of dimension $n-3$ whose points parametrize all
metric trees with $n$ labeled leaves. The cones distinguish different combinatorial types of metric trees.
The defining polynomials of this (tropical) variety are 
the $\binom{n}{4}$ Pl\"ucker quadrics $p_{ij} p_{k\ell} - p_{ik} p_{j\ell} + p_{i\ell} p_{jk}$.
These quadrics are the $4 \times 4$-subpfaffians of
a skew-symmetric $n \times n$-matrix, and they form
a {\em tropical basis} for $\mathcal{M}_{0,n}$. 
The {\em tropical compactification} defined by this fan
is the moduli space $\overline{\mathcal{M}}_{0,n}$ of $n$-pointed stable rational curves.
The picture for $n=5$ is delightful:
the tropical surface  ${\rm trop}(\mathcal{M}_{0,5})$ is the cone
over the {\em Petersen graph}, with vertices labeled by the
$10$ Pl\"ucker coordinates $p_{ij}$ as in Figure~\ref{fig:petersen}.

\begin{figure}[h]
\begin{center}
\includegraphics[width=0.3\textwidth]{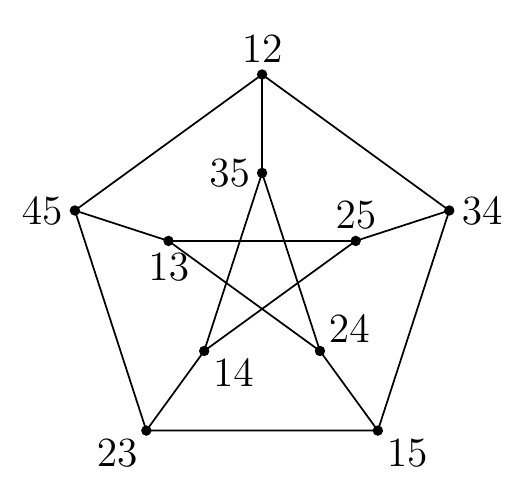}
\end{center}
\vspace{-0.3in}
  \label{fig:petersen}
\caption{The Petersen graph represents the tropicalization of $\mathcal{M}_{0,5}$.}
\end{figure}

A related example is the universal family $\mathcal{A}(5)$  over the modular curve $X(5)$.
The relevant combinatorics goes back to Felix Klein and his
famous 1884 lectures on the icosahedron \cite{klein}.
Following Fisher \cite{Fis}, the surface $\mathcal{A}(5)$ sits in
  $\mathbb{P}^1 \times \PP^4$ and has the Pfaffian representation
\begin{equation}
\label{eq:pfaff5}
 {\rm rank} \begin{bmatrix}
0 & -a_1x_1 & -a_2 x_2 & a_2 x_3 & a_1 x_4 \\
 a_1 x_1 & 0 & -a_1 x_3 & -a_2 x_4  & a_2 x_0 \\
a_2 x_2 & a_1 x_3 & 0 & -a_1 x_0 & -a_2 x_1 \\
-a_2 x_3 & a_2 x_4 & a_1 x_0 & 0 & -a_1 x_2 \\
-a_1 x_4 & -a_2 x_0 & a_2 x_1 & a_1 x_2 & 0 
\end{bmatrix} \, \leq \,\, 2 . 
\end{equation}
The base of this family is $ \mathbb{P}^1$ with coordinates $(a_1:a_2)$.
The tropical surface ${\rm trop} (\mathcal{A}(5))$ is
a fan in $\mathbb{TP}^1 \times \mathbb{TP}^4$, which is
 combinatorially the Petersen graph in Figure~\ref{fig:petersen}.
 The central fiber,
over the vertex of $\mathbb{TP}^1$
given by ${\rm val}(a_1) = {\rm val}(a_2)$,
is the $1$-dimensional fan with rays
$e_0,e_1,e_2,e_3,e_4$. These
correspond to the edges
34-25, 12-35, 45-13, 23-14, 15-24.
For ${\rm val}(a_1) {<} {\rm val}(a_2)$,
the fiber is given by the pentagon 12-34-15-23-45-12 
with these rays attached. For
${\rm val}(a_1) {>} {\rm val}(a_2)$, it is the pentagram
35-14-25-13-24-35  with the five rays.
Each of the edges has multiplicity $5$.
The map from ${\rm trop} (\mathcal{A}(5))$ onto $\mathbb{TP}^1$
 is visualized in Figure~\ref{fig:pentagonpentagram}.

\begin{figure}[h]
\begin{center}
\vspace{-0.1in}
\includegraphics[width=0.5\textwidth]{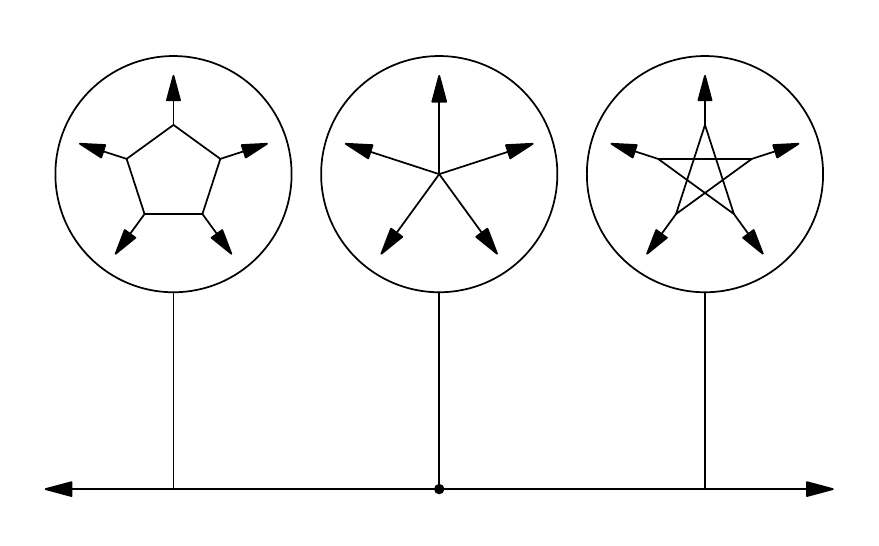}
\end{center}
\vspace{-0.35in}
  \label{fig:pentagonpentagram}
\caption{The universal family of tropical elliptic normal curves of degree $5$.}
\end{figure}

The discriminant of our family $\mathcal{A}(5) \rightarrow \PP^1$ is the  binary form
\begin{equation}
\label{eq:binary12}
 a_1^{11} a_2 \,- \,11 a_1^6 a_2^6 \,-\,a_1 a_2^{11}, 
 \end{equation}
whose $12$ zeros represent Klein's icosahedron.
The modular curve $X(5)$ is $\PP^1$ minus these $12$ points.
For each $(a_1:a_2) \in X(5)$, the condition (\ref{eq:pfaff5})
defines  an elliptic normal curve in $\PP^4$.

Throughout this paper we work 
over an algebraically closed field $K$ of characteristic $0$. 
Our notation and conventions
regarding tropical geometry follow \cite{MS}. For simplicity of exposition,
we identify the tropical projective space $\mathbb{TP}^n$ with its open part
 $\R^{n+1} /\R (1,1,\ldots,1)$.
 
The adjective ``classical'' in our title has  two meanings.
{\em Classical as opposed to tropical} refers to moduli spaces that
are defined over fields, the usual setting of algebraic geometry.
The foundations for tropicalizing such schemes and stacks are currently being
developed, notably in the work of
Abramovich {\it et al.}~\cite{ACP} and
Baker {\it et al.}~\cite{BPR} (see \cite{abramovich} for a survey). These rest on the
connection to nonarchimedean geometry.
{\em Classical as opposed to modern} refers to moduli spaces
that were known in the 19th century.  We focus here on the varieties
featured in Hunt's book \cite{hunt}, notably the
Segre cubic, the Igusa quartic, the Burkhardt quartic,
and their universal~families. We shall also
revisit the work on tropical del Pezzo surfaces by
Hacking {\it et al.} in \cite{HKT} and explain how this relates to the 
tropical G\"opel variety of \cite[\S 9]{RSSS}.

Each of our moduli spaces admits a high-dimensional symmetric embedding of the form
\begin{equation}
\label{eq:twomaps}  \PP^d \,\, \buildrel{{\rm linear}}\over{\hookrightarrow}\,\, \PP^m\,\,
\buildrel{{\rm monomial}} \over{\dashrightarrow} \,\, \PP^n. 
\end{equation}
The coordinates of the first map are the
linear forms defining the $m{+}1$ hyperplanes in a
complex reflection arrangement $\mathcal{H}$ in $\PP^d$,
while the coordinates of the second map are monomials
that encode the symplectic geometry of a finite vector space.
The relevant combinatorics  rests on the representation theory
developed in \cite{GS, GSW}.
Each of our moduli spaces is written as the image of a map
 (\ref{eq:twomaps}) whose coordinates are monomials in linear forms,
 and hence the formula in~\cite[Theorem 3.1]{DFS}  expresses
its tropicalization using the matroid structure of $\mathcal{H}$.

Our warm-up example, the modular curve $X(5)$,
fits the pattern  (\ref{eq:twomaps}) for $d=1, m=11$ and $n=5$.
Its arrangement $\mathcal{H} \subset \PP^1$ is
the set of $12$ zeros of (\ref{eq:binary12}), but now identified with the
complex reflection arrangement ${\rm G}_{16}$ as in
\cite[\S 2.2]{GS}. If we factor 
(\ref{eq:binary12}) into six quadrics,
$$ \bigr(a_1 a_2\bigr) \cdot \prod_{i=1}^5 \bigl((\gamma^{5-i} a_1+(\gamma{+}\gamma^4) a_2) (\gamma^ia_1+(\gamma^2{+}\gamma^3)a_2 \bigr) ,$$
where $\gamma$ is a primitive fifth root of unity, then these
define the coordinates of $ \,\PP^{11}
\buildrel{{\rm monomial}} \over{\dashrightarrow} \PP^5.  $
The image is a quadric in a plane in $\PP^5$,
and $X(5)$ is now its intersection with the torus $\mathbb{G}_m^5$.
The symmetry group ${\rm G}_{16}$ acts on $\PP^5$
by permuting the six homogeneous coordinates.
The tropical modular curve ${\rm trop}(X(5))$ 
is the standard one-dimensional fan  in $\mathbb{TP}^5$,
with multiplicity five and  pentagonal fibers as above. But now
the full symmetry group acts on
the surface $\mathcal{A}(5) \subset \PP^5 \times \PP^4$
and the corresponding tropical surface by permuting coordinates.

\smallskip

We next discuss the organization of this paper.
In Section~\ref{sec:segrecubic} we study the  Segre cubic and the Igusa quartic, in their symmetric embeddings into $\mathbb{P}^{14}$ and
$\mathbb{P}^9$, respectively. We show that the corresponding
tropical variety is the space of phylogenetic trees on six taxa,
and we determine the universal family of tropical Kummer surfaces over that base.
In Section~\ref{sec:burkhardt} we study the Burkhardt quartic in its symmetric embedding
in $\mathbb{P}^{39}$, and, over that base, we compute the universal
family of abelian surfaces in $\PP^8 $  along with
their associated tricanonical curves of genus~$2$.
In Section~\ref{sec:tropicalization} we compute the Bergman fan
of the complex reflection arrangement $\mathrm{G}_{32}$
and from this we derive the tropical Burkhardt quartic in $\mathbb{TP}^{39}$.
The corresponding tropical compactification is shown to coincide with 
the Igusa desingularization of the Baily--Borel--Satake compactification of $\mathcal{A}_2(3)$.
In Section~\ref{sec:genus2moduli} we relate our findings to the
abstract tropical moduli spaces of \cite{BMV, Cha}.
Figure \ref{table-tropical-moduli}  depicts the resulting correspondence between
trees on six taxa,  metric graphs of genus 2, and cones in the 
tropical Burkhardt quartic.
In Section~\ref{sec:delpezzo} we study the reflection arrangements of types
$\mathrm{E}_6$ and $\mathrm{E}_7$, and we show how they lead to the
tropical moduli spaces of marked del Pezzo surfaces
constructed by Hacking, Keel and Tevelev \cite{HKT}.
For $\mathrm{E}_7$ we recover the tropical G\"opel variety of \cite[\S 9]{RSSS}.
This is a six-dimensional fan which serves as the universal family of tropical cubic surfaces.

\subsection*{Acknowledgements}

Steven Sam was supported by a Miller Research Fellowship 
at UC Berkeley.
Qingchun Ren and Bernd Sturmfels were supported by the
National Science Foundation (DMS-0968882) and DARPA (HR0011-12-1-0011).
We thank Florian Block, Dustin Cartwright, Melody Chan, Diane Maclagan, Sam Payne
and Jenia Tevelev for helpful discussions.
We are especially grateful to Gus Schrader for his contributions to 
the material in Section~\ref{sec:burkhardt}.

\section{Segre Cubic, Igusa Quartic, and Kummer Surfaces}
\label{sec:segrecubic}

The moduli spaces in this section are based on
the hyperplane arrangement in $\PP^4$ associated
with the reflection representation of the symmetric group
$\Sigma_6$. It consists of the $15$ hyperplanes
\begin{equation}
\label{eq:xixj} \qquad   x_i - x_j  \,\, = \,\, 0 \quad \qquad (1 \leq i< j \leq 6). 
\end{equation}
Here $\PP^4$ is the projectivization of the
$5$-dimensional  vector space $K^6/K(1,1,1,1,1,1)$.
The $15$ linear forms in (\ref{eq:xixj}) define the map
$\, \PP^4  \buildrel{{\rm linear}}\over{\hookrightarrow} \PP^{14} \,$
whose image is the $4$-dimensional subspace
${\rm Cyc}_4$ of $\PP^{14}$ that is defined by the linear equations
$z_{ij} - z_{ik} + z_{jk} = 0$
for $1 \leq i < j < k \leq 6$.

The corresponding tropical linear space ${\rm trop}({\rm Cyc}_4)$, with the coarsest fan structure, is isomorphic to both the moduli space of equidistant (rooted) phylogenetic trees with $6$ vertices and the moduli space of (unrooted) phylogenetic trees with $7$ vertices.
The former was studied by Ardila and Klivans in \cite[\S 4]{AK}. 
They develop the correspondence between ultrametrics and equidistant
phylogenetic trees in \cite[Theorem 3]{AK}. The latter is a tropicalization of the Grassmannian $ {\rm Gr}(2,7)$ as described in \cite[\S 4]{SS}.
From the combinatorial description given there one derives the face numbers below:

\begin{lemma}
The tropical linear space ${\rm trop}({\rm Cyc}_4)$
is the space of ultrametrics on $6$ elements,
or, equivalently, the space of equidistant phylogenetic trees on $6$ taxa. 
It is a fan over a three-dimensional simplicial complex
with 
$56$ vertices,
$ 490$ edges,
$ 1260$ triangles
and $945$ tetrahedra.
\end{lemma}

We now define our two modular threefolds by way of
a monomial map from $\PP^{14}$ to another space
$\PP^n$. The homogeneous coordinates on that $\PP^n$
will be denoted $m_0, m_1, \ldots,m _n$,
so as to highlight that they can be identified with 
certain modular forms, known as
theta constants.

The {\em Segre cubic} $\mathcal{S}$  is the closure of the image of ${\rm Cyc}_4$
under $\,\PP^{14} \buildrel{{\rm monomial}} \over{\dashrightarrow} \PP^{14}\,$ given by
\begin{equation}
\label{eq:segremap} 
\begin{matrix}
(z_{12} z_{34} z_{56} : z_{12} z_{35} z_{46} : z_{12} z_{36} z_{45} : z_{13} z_{24} z_{56} : z_{13} z_{25} z_{46} : z_{13} z_{26} z_{45} : z_{14} z_{23} z_{56} :\\ 
z_{14} z_{25} z_{36} : z_{14} z_{26} z_{35} : z_{15} z_{23} z_{46} : z_{15} z_{24} z_{36} : z_{15} z_{26} z_{34} : z_{16} z_{23} z_{45} : z_{16} z_{24} z_{35} : z_{16} z_{25} z_{34}).
\end{matrix}
\end{equation}
The prime ideal of $\mathcal{S}$ is generated by $10$ linear trinomials, like  $m_0 - m_1 + m_2$, that 
come from Pl\"ucker relations among the $ x_i - x_j$, and one cubic binomial such as
$\,m_0 m_7 m_{12}-m_2 m_6 m_{14}$. For a graphical representation of this ideal we refer to
Howard {\em et al.}~\cite[(1.2)]{HMSV}: for the connection, note that the monomials in \eqref{eq:segremap} naturally correspond to perfect matchings of a set of size $6$, which are the colored graphs in \cite{HMSV}.

To see that this is the same as the classical definition of the Segre cubic, the reader can jump ahead to \eqref{eq:segre_rst} and \eqref{eq:from_m_to_rst}.

The {\em Igusa quartic} $\mathcal{I}$   is the closure of the image of ${\rm Cyc}_4$ under
$\PP^{14} \buildrel{{\rm monomial}} \over{\dashrightarrow} \PP^9$ given~by
\begin{small}
$$
\label{eq:igusamap} 
\begin{matrix} (
z_{12} z_{13} z_{23} z_{45} z_{46} z_{56} \!:\! 
z_{12} z_{14} z_{24} z_{35} z_{36} z_{56} \!:\! 
z_{12} z_{15} z_{25} z_{34} z_{36} z_{46} \!:\!
z_{12} z_{16} z_{26} z_{34} z_{35} z_{45}\!:\!
z_{13} z_{14} z_{34} z_{25} z_{26} z_{56}: \\
z_{13} z_{15} z_{35} z_{24} z_{26} z_{46}\! : \! 
z_{13} z_{16} z_{36} z_{24} z_{25} z_{45}\! : \!
z_{14} z_{15} z_{45} z_{23} z_{26} z_{36} \! : \!
 z_{14} z_{16} z_{46} z_{23} z_{25} z_{35}\! : \!
 z_{15} z_{16} z_{56} z_{23} z_{24} z_{34} )
\end{matrix}
$$
\end{small}
The prime ideal of $\mathcal{I}$ is generated by the five linear forms in
the column vector
\begin{equation}
\label{eq:5by5}
 \begin{pmatrix}
 0 & m_0 & m_1 & m_2 & m_3 \\
m_0&  0& m_4& m_5& m_6 \\
m_1& m_4&  0& m_7& m_8 \\
m_2& m_5& m_7&  0& m_9 \\
m_3& m_6& m_8& m_9& 0
\end{pmatrix}
\cdot 
\begin{pmatrix} \phantom{-}1 \,\\ -1 \,\\
\phantom{-} 1\, \\ -1 \,\\ \phantom{-}1\, \end{pmatrix}
\end{equation}
together with any of the $4 \times 4$-minors of the 
symmetric $5 \times 5$-matrix in (\ref{eq:5by5}).
The linear forms (\ref{eq:5by5}) come from
Pl\"ucker relations of degree $(1,1,1,1,1,1)$ 
on ${\rm Gr}(3,6)$. We note that $m_0,\ldots,m_9$ can be written
in terms of theta functions by Thomae's theorem \cite[\S VIII.5]{DO}.

To see that this is the usual Igusa quartic, one can calculate the projective dual of the quartic hypersurface we have just described and verify that it is a cubic hypersurface whose singular locus consists of $10$ 
nodes. The Segre cubic is the unique cubic in $\PP^4$ with $10$ nodes.

A key ingredient in the study of modular varieties is the symplectic combinatorics of finite vector spaces. Here we consider the binary space $\mathbb{F}_2^4$ with the symplectic form
\begin{equation}
\label{eq:innerproduct}
 \langle x,y \rangle \,\, = \,\,  x_1 y_3 + x_2 y_4 - x_3 y_1 - x_4 y_2 .
 \end{equation}
 We fix the following bijection between the $15$  hyperplanes 
(\ref{eq:xixj})
 and the vectors in $\mathbb{F}_2^4 \backslash \{0\}$:
\begin{equation}
\label{eq:bijection15}
\!\!
\begin{matrix}
z_{12} & z_{13} &  z_{14} &  z_{15} &  z_{16} &  z_{23} &  z_{24} &  z_{25} &  z_{26} &  z_{34} &  
z_{35} &  z_{36} &  z_{45} &  z_{46} &  z_{56} \\
\!\! u_{0001}  \! & \!\! u_{1100}  \! & \!\! u_{1110}  \! & \!\!
u_{0101}  \! &  \!\! u_{0110}  \! & \!\! u_{1101}  \! & \! u_{1111}  \! &  \! u_{0100}  \! &
\! u_{0111}  \! & \! u_{0010} 
 \! & \! u_{1001}  \! &  \! u_{1010}  \! &\! u_{1011}  \! & \! u_{1000}  \! & \!  u_{0011}
\end{matrix}
\end{equation}
This bijection has the property that two
vectors in $\mathbb{F}_2^4 \backslash \{0\}$ 
are perpendicular with respect to (\ref{eq:innerproduct})
if and only if the corresponding elements of the root system 
$\mathrm{A}_5$ are perpendicular. Combinatorially, this means that
 the two pairs of indices are disjoint.
There are precisely $35$ two-dimensional subspaces $L$ in
 $\mathbb{F}_2^4$. Of these planes $L$, precisely $15$ are isotropic,
 which means that $L = L^\perp$. The other $20$ planes naturally
 come in pairs $\{L, L^\perp\}$. Each plane
 is a triple in $\mathbb{F}_2^4 \backslash \{0\}$ 
 and we write it as a cubic monomial 
 $z_{ij} z_{k\ell} z_{mn}$.
Under this identification, the parametrization (\ref{eq:segremap}) of the Segre cubic $\mathcal{S}$ is given by the $15$ isotropic planes $L$, while that  of the Igusa quartic $\mathcal{I}$ is given by the $10$ pairs $L \cdot L^\perp$ of non-isotropic planes in $\mathbb{F}_2^4$.

The symplectic group $ \mathbf{Sp}_4(\mathbb{F}_2)$ 
consists of all linear automorphisms of $\mathbb{F}_2^4$ that
preserve the symplectic form (\ref{eq:innerproduct}). As an abstract group it
is isomorphic to the symmetric group on six letters:
 \begin{align} \label{eqn:level2isom}
\mathbf{Sp}_4(\mathbb{F}_2) \,\,\cong \,\, \Sigma_6.
\end{align}
This group isomorphism is made explicit by the bijection
(\ref{eq:bijection15}).
  
Let $\mathcal{M}_2(2)$ denote the moduli space of
smooth curves of genus $2$ with a level $ 2$ structure.
In light of the isomorphism (\ref{eqn:level2isom}),  a level $2$
structure on a genus $2$ curve $C$  is an ordering of its six Weierstrass points,
and this corresponds to the choice of six labeled points
on the projective line $\PP^1$. The latter choices are parametrized
by the moduli space $\mathcal{M}_{0,6}$.
In what follows,
we consider the {\em open Segre cubic} 
$\,\mathcal{S}^{\circ}  =  \mathcal{S} \backslash \{m_0 m_1 \cdots m_{14}  =0\}\,$
inside the torus  $\mathbb{G}_m^{14} \subset \PP^{14}\,$ and
 the {\em open Igusa quartic} 
$\,\mathcal{I}^{\circ}  =  \mathcal{I} \backslash \{m_0 m_1 \cdots m_9  =0 \}\,$
  inside the torus  $\mathbb{G}_m^{9} \subset \PP^9$.

\begin{proposition}
\label{prop:identification}
We have the following identification of
three-dimensional moduli spaces:
\begin{equation}
\label{eq:SIMM}
\mathcal{S}^{\circ} \,= \,\mathcal{I}^{\circ} \,=\,
\mathcal{M}_2(2) = \mathcal{M}_{0,6}.
\end{equation}
\end{proposition}

\begin{proof}
We already argued the last equation.
The first equation is the isomorphism between
the open sets $D$ and $D'$ in the proof of
\cite[Theorem 3.3.11]{hunt}. A nice way to 
see this isomorphism is that the kernels
of our two monomial maps coincide (Lemma~\ref{lem:kernel}).
The middle equation  follows from the
last part of \cite[Theorem 3.3.8]{hunt},
which concerns the Kummer functor $\mathbf{K}_2$.
For more information on  the modular interpretations
of $\mathcal{S}$ and $\mathcal{I}$ see \cite[\S VIII]{DO}.
\end{proof}

The Kummer surface associated to a point
in $\mathcal{I}^{\circ}$ is the intersection of the Igusa quartic
$\mathcal{I}$ with the tangent space at that point,
 by \cite[Theorem 3.3.8]{hunt}.
We find it convenient to express that Kummer surface
in terms of the corresponding point in $\mathcal{S}^{\circ}$.
Following Dolgachev and Ortland \cite[\S IX.5, Proposition~6]{DO},
we write the defining equation of the Segre cubic
$\mathcal{S}$ as
\begin{equation}
\label{eq:segre_rst}
  16 r^3 - 4 r (s_{01}^2+s_{10}^2+s_{11}^2) + 4 s_{01} s_{10} s_{11} + r t^2 \,\,= \,\, 0. 
\end{equation}
The embedding of the $\PP^4$ 
with coordinates $(r \!:\!s_{01}\!:\!s_{10}\!:\!s_{11}\!:\!t) $ into our $\PP^{14}$ can be written as
\begin{equation}
\label{eq:from_m_to_rst}
\begin{matrix}
 r= m_0 ,\qquad
 s_{01} =  2m_0 - 4m_1 ,\qquad
 s_{10} =  2m_0 - 4m_3, \qquad \\ \qquad
 s_{11} = 4m_4 - 2m_0  - 4m_7 ,\qquad
 t = 8(m_1+ m_3 - m_0  - m_4 -   m_7).
 \end{matrix}
\end{equation}
This does not pick out an $\Sigma_6$-equivariant embedding of the space spanned by the $r, s_{ij}$ in the permutation representation of the $m_i$, but it has the advantage of giving short expressions.
Fixing Schr\"odinger coordinates
$(x_{00}\!:\! x_{01 \!}\!:x_{10}\!: \! x_{11})$ on $\PP^3$, the 
 Kummer surface is now given~by
\begin{equation}
 \label{eq:kummer}
 \begin{matrix}
     r (x_{00}^4+x_{01}^4+x_{10}^4+x_{11}^4)
+  s_{01} (x_{00}^2 x_{01}^2+x_{10}^2 x_{11}^2)
+  s_{10} (x_{00}^2 x_{10}^2+x_{01}^2 x_{11}^2) \\
+  s_{11} (x_{00}^2 x_{11}^2 + x_{01}^2 x_{10}^2)
+    t  (x_{00} x_{01} x_{10} x_{11}) \quad = \quad 0.
\end{matrix}
\end{equation}
This equation is the determinant of the
$5 \times 5$-matrix in
\cite[Example 1.1]{RSSS}. Its lower
$4 \times 4$-minors satisfy (\ref{eq:segre_rst}).
Our notation is consistent with that for the Coble quartic in \cite[(2.13)]{RSSS}.

We now come to the tropicalization
of our three-dimensional moduli spaces.
We write $e_{12}, e_{13}, \ldots, e_{56}$ for
the unit vectors in $\mathbb{TP}^{14} = 
\mathbb{R}^{15}/\mathbb{R} (1,1,\ldots,1)$.
These correspond to our coordinates
$z_{12}, z_{13}, \ldots, z_{56}$ on the
$\PP^{14}$ which contains ${\rm Cyc}_4 \simeq \PP^4$.
The $56$ rays of the Bergman fan ${\rm trop}({\rm Cyc}_4)$
are indexed by proper subsets $\sigma \subsetneqq \{1,2,3,4,5,6\}$
with $|\sigma| \geq 2$. They are
$$ E_\sigma \,\, \, = \,\, \sum_{\{i,j\} \subseteq \sigma} e_{ij}. $$
Cones in ${\rm trop}({\rm Cyc}_4)$ are spanned by
collections of $E_\sigma$ whose indices $\sigma$ are nested or disjoint.

Let $A_{\rm segre}$ denote the $15 \times 15$-matrix with entries in $\{0,1\}$ that represents the tropicalization of the monomial map (\ref{eq:segremap}). The columns of $A_{\rm segre}$ are indexed by (\ref{eq:bijection15}). The rows of $A_{\rm segre}$ are indexed by tripartitions of $\{1,2,\ldots,6\}$,
or by isotropic planes in $\mathbb{F}_2^4$.
An entry is $1$ if the pair that indexes the column
appears in the tripartition that indexes the row, 
or, equivalently, if the line of $\mathbb{F}_2^4$
that indexes the column is contained in the plane that indexes the row.
Note that each row and each column of $A_{\rm segre}$ has
precisely three nonzero entries.

We similarly define the $10 \times 15$-matrix $A_{\rm igusa}$
with entries in $\{0,1\}$ that represents the monomial
map for the Igusa quartic. Its rows have six nonzero
entries and its columns have four nonzero entries.
The column labels of $A_{\rm igusa}$ are the same
as those of $A_{\rm segre}$. The rows
are now labeled by bipartitions of $\{1,2,\ldots,6\}$, or by pairs of non-isotropic planes in $\mathbb{F}_2^4$.
 
\begin{lemma} \label{lem:kernel}
The matrices $A_{\rm segre}$ and $A_{\rm igusa}$ have the same kernel.
This kernel is the $5$-dimensional subspace spanned by the vectors $E_\sigma - E_{\sigma^c}$ where $\sigma$ runs over triples in $\{1,2,\ldots,6\}$.
\end{lemma}

This lemma can be proved by a direct computation.
The multiplicative version of this fact  implies the
identity $\,\mathcal{S}^{\circ} = \mathcal{I}^{\circ}\,$ as seen in
Proposition \ref{prop:identification}. We have the following result.

\begin{theorem} \label{thm:tropsegrecubic}
The tropical Segre cubic ${\rm trop}(\mathcal{S}) $ in $\mathbb{TP}^{14} $ 
is the image  of  ${\rm trop}({\rm Cyc}_4)$ under the linear map
 $A_{\rm segre}$. The tropical Igusa quartic ${\rm trop}(\mathcal{I}) $ in $\mathbb{TP}^{9} $
is the image  of  ${\rm trop}({\rm Cyc}_4)$ under the linear map
$A_{\rm igusa}$. These two $3$-dimensional fans are affinely isomorphic to each other,
but all maximal cones of  ${\rm trop}(\mathcal{I}) $ come with multiplicity two.
The underlying simplicial complex has
$25$ vertices, $105$ edges and $105$ triangles. This
is the tree space ${\rm trop}(\mathcal{M}_{0,6})$.
\end{theorem}

\begin{proof}
The fact that we can compute the tropicalization of
the image of a linear space under a monomial map
by just applying the tropicalized monomial map
$A_\bullet$ to the
Bergman fan is \cite[Theorem 3.1]{DFS}.
The fact that the two tropical threefolds are
affinely isomorphic  follows immediately from Lemma~\ref{lem:kernel}.
To analyze the combinatorics of this common image fan, 
we set $E_\sigma $ to be the zero vector
when $\sigma = \{i\}$ is a singleton.
With this convention, we have
$$ A_{\rm segre} E_\sigma = A_{\rm segre} E_{\sigma^c} \quad
\hbox{and} \quad
A_{\rm igusa} E_\sigma = A_{\rm igusa} E_{\sigma^c} 
 $$
 for all proper subsets $\sigma$ of $\{1,2,\ldots,6\}$.
  We conclude that the $56 = 15 + 20 + 15 + 6$ rays of the
Bergman fan ${\rm trop}({\rm Cyc}_4)$ get mapped to
$25 = 15 + 10$ distinct rays in the image fan.

The cones in ${\rm trop}({\rm Cyc}_4)$ 
correspond to equidistant trees, that is,
\underline{rooted} metric trees on six taxa. Combinatorially,
our map corresponds to removing the root from the tree, so the
cones in the image fan correspond to 
\underline{unrooted} metric trees on six taxa.
Specifically, each of the $945$ maximal cones of
${\rm trop}({\rm Cyc}_4)$ either has one ray 
$E_{\{i,j,k,\ell,m\}}$ that gets mapped to zero,
or it has two rays $E_\sigma$
and $E_{\sigma^c}$ that become identified.
Therefore its image is three-dimensional.
Our map takes the $945$ simplicial cones of dimension $4$ in
${\rm trop}({\rm Cyc}_4)$
onto the $105$ simplicial cones of dimension $3$,
one for each unrooted tree. The fibers involve precisely nine
cones because each trivalent tree on six taxa has
nine edges, each a potential root location.
Combinatorially, nine rooted trivalent trees map to the same
unrooted  tree.

It remains to analyze the multiplicity of each maximal cone in the image.
The $105$ maximal cones in ${\rm trop}(\mathcal{S})$
all have multiplicity one, while the corresponding cones
in ${\rm trop}(\mathcal{I})$ have multiplicity two. 
We first found this by a direct calculation using the software ${\tt gfan}$ \cite{jensen},
starting from the homogeneous ideals of   $\mathcal{S}$ and ${\mathcal{I}}$ described above.
It can also be seen by examining the images of the rays $E_\tau$
under each matrix $A_\bullet$ modulo the line spanned by the
 vector $(1,1,\ldots,1)$.
Each of the $15$ vectors  $A_{\rm igusa} E_{ij}$ is the
sum of four unit vectors in $\mathbb{TP}^9$, while
the $10$ vectors  $A_{\rm igusa} E_{ijk}$
are the ten unit vectors multiplied by the factor $2$.
\end{proof}

We next discuss the tropicalization of the universal family 
of Kummer surfaces over $\mathcal{S}^{\circ}$.
This is the hypersurface in $\,\mathcal{S}^{\circ} \times \mathbb{P}^3$
defined by the equation  (\ref{eq:kummer}). 
The tropicalization of this hypersurface
is a five-dimensional fan whose fibers over
the tree space $\,{\rm trop}(\mathcal{S})\,$
are the  tropical Kummer surfaces in 
$\mathbb{TP}^3$.
We computed this fan from the equations using
{\tt gfan} \cite{jensen}.

\begin{proposition}
The tropicalization of the universal Kummer surface in the coordinates $((m_0:m_1:\dotsb{}:m_{14}),(x_{00}:x_{01}:x_{10}:x_{11}))$ is a $5$-dimensional polyhedral fan
in $\mathbb{TP}^{14} \times \mathbb{TP}^3$. This fan has $56$ rays and $1536$ maximal cones,
and its f-vector is  $\,(56, 499 ,1738, 2685, 1536)$.
\end{proposition}

Instances of tropical Kummer surfaces can be obtained by slicing the above fan with fixed values of the $15$ tropical $m$ coordinates. Figure~\ref{fig:kummer-snowflake} shows the tropicalization of a Kummer surface
 over a snowflake tree  (Type (7) in Table \ref{table-tropical-moduli}).
It consists of $30$ two-dimensional polyhedra, $24$ unbounded and
$6$ bounded. The latter $6$ form the facets of a parallelepiped.

\begin{figure}[h!]
\centering
\includegraphics[scale=0.5]{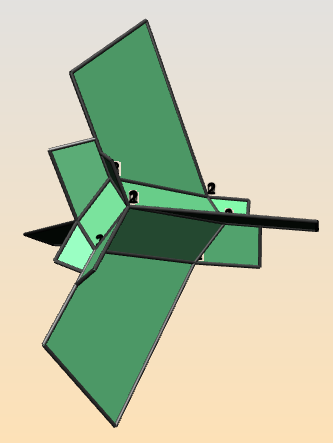}
\label{fig:kummer-snowflake}
\caption{Tropicalization of a Kummer surface over a snowflake tree.}
\end{figure}

Figure~\ref{fig:kummer-caterpillar} shows a tropical Kummer surface over
a caterpillar tree (Type (6) in Table \ref{table-tropical-moduli}).
It consists of $33$ two-dimensional polyhedra, $24$ bounded and
$9$ bounded. The latter $9$ polygons form a subdivision of a flat octagon.
These two pictures were drawn using {\tt polymake} \cite{GJ}.

\begin{figure}[h!]
\centering
\includegraphics[scale=0.48]{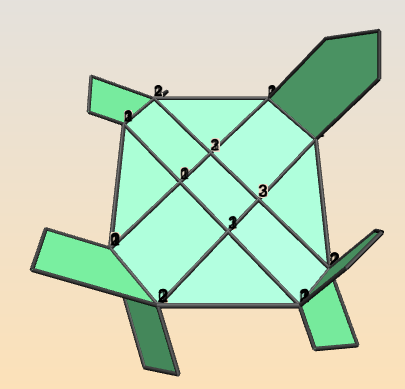}
\label{fig:kummer-caterpillar}
\caption{Tropicalization of a Kummer surface over a caterpillar tree.}
\end{figure}

On each Kummer surface we could now
identify a tree that represents the bicanonical image of the associated
genus $2$ curve. Classically, one obtains a double quadric with six distinguished points
by intersecting with any of the planes in the $16_6$ configuration
\cite[(1.2)]{RSSS}.

The tropical variety described in Theorem \ref{thm:tropsegrecubic}
defines the {\em tropical compactification} $\overline{\mathcal{S}}$ of the
Segre cubic $\mathcal{S}$. By definition, the
threefold $\overline{\mathcal{S}}$ is the closure
of $\mathcal{S}^{\circ}$ in the toric variety determined by the
given fan structure on ${\rm trop}(\mathcal{S})$.
For details, see Tevelev's article \cite{Tev}.

This tropical compactification of our moduli space (\ref{eq:SIMM})
is intrinsic. To see this, we recall that the {\em intrinsic torus}
of a very affine variety $X \subset  \mathbb{G}_m^n$
is the torus whose character lattice is the finitely generated
multiplicative free abelian group $K[X]^*/K^*$.
The following lemma can be used to find the intrinsic torus for 
each of the very affine varieties in this paper.

\begin{lemma} \label{lem:intrinsictorusimage}
Let $m \colon T_1 \to T_2$ be a monomial map of tori and $U \subset T_1$
a subvariety embedded in its intrinsic torus. Then $\overline{m(U)} \subset m(T_1)$ is the embedding of $\overline{m(U)}$ in its intrinsic torus.
\end{lemma}

\begin{proof}
Choose identifications $K[T_1] = K[x_1^{\pm}, \dots, x_r^{\pm}]$ and $K[T_2] = K[y_1^{\pm}, \dots, y_s^{\pm}]$. By assumption, the pullback $m^*(y_i)$ is a monomial in the $x_j$, which we call $z_i$. We have an injection 
of rings $m^* \colon K[\overline{m(U)}] \subset K[U]$, and hence we get an induced injection of groups $\phi \colon K[\overline{m(U)}]^* / K^* \subset K[U]^* / K^*$. Since $K[\overline{m(U)}]$ is generated by the $y_i$, we conclude that $m^*(K[\overline{m(U)}])$ is contained in the subalgebra $K[z_1^{\pm}, \dots, z_s^{\pm}] \subset K[U]$. Pick $f \in K[\overline{m(U)}]^* / K^*$. Since $U$ is embedded in its intrinsic torus,
we have $\phi(f) = z_1^{d_1} \cdots z_s^{d_s}$ for some $d_i \in \Z$. So $\phi(y_1^{d_1} \cdots y_s^{d_s}) = \phi(f)$ and since $\phi$ is injective, we conclude that $f = y_1^{d_1} \cdots y_s^{d_s}$.
\end{proof}

The embedding of the Segre cubic $\mathcal{S}$ into 
the $9$-dimensional toric variety given by (\ref{eq:segremap})
satisfies the hypotheses of Lemma \ref{lem:intrinsictorusimage}.
Indeed, $\mathcal{S}^{\circ}$
 is the image of the complement of a hyperplane arrangement
under a monomial map, and, by \cite[\S 4]{Tev},
the intrinsic torus of an essential arrangement of 
$n$ hyperplanes in $\PP^r$ is $\mathbb{G}_m^{n-1}$.
The same argument works for all 
moduli spaces studied  in this paper.
 That the ambient torus $\mathbb{G}_m^9$ is intrinsic for
 the open Segre cubic  $\mathcal{S}^{\circ} $
can also be seen from the fact that the
 $15$ boundary divisors $\mathcal{S} \cap \{m_i=0\}$ are irreducible.
Indeed, by \cite[\S 3.2.1]{hunt},
they are projective planes $\mathbb{P}^2$.
Each of the ten singular points of $\mathcal{S}$ lies
on six of these planes,
so each boundary plane contains four singular points.

From the combinatorial description above we infer the following summary of the situation.

\begin{corollary}
The tropical compactification of the open Segre cubic $\mathcal{S}^{\circ}$,
and hence of the other
moduli spaces in \eqref{eq:SIMM}, is the
Deligne--Mumford compactification  $\overline{\mathcal{M}}_{0,6}$. 
This threefold is the blow-up of the $10$ singular points of $\mathcal{S}$, or of the $15$ singular lines of the Igusa quartic~$\mathcal{I}$.
\end{corollary}

The second sentence is Theorem 3.3.11 in Hunt's book \cite{hunt}.
The first is a special case of \cite[Theorem 1.11]{HKT}.
Our rationale for giving a detailed equational derivation of the familiar manifold $\overline{\mathcal{M}}_{0,6}$ is that it sets the stage for our primary example in the next section.

\section{Burkhardt Quartic and Abelian Surfaces}
\label{sec:burkhardt}

The Burkhardt quartic is a rational quartic threefold in $\PP^4$. It can be characterized as the unique quartic hypersurface in $\PP^4$ with  the maximal number $45$ of nodal singular points 
\cite{DJSV}.
It compactifies the moduli space $\mathcal{M}_2(3)$
of genus 2 curves with level 3 structure \cite{DL,FSM,GS,hunt}. We identify $\mathcal{M}_2(3)$ with a subvariety of $\mathcal{A}_2(3)$, the moduli space of principally polarized abelian surfaces with level 3 structure, by sending a 
smooth curve to its Jacobian.

All constructions in this section can be
carried out over any field $K$
of characteristic other than $2$ or $3$, provided
$K$ contains a primitive third root of unity $\omega$.
In the tropical context,  $K$ will be a field with a valuation.
For details on arithmetic issues see Elkies' paper \cite{Elk}.

We realize the Burkhardt quartic
as the image of a rational map that is given as a composition
$\,  \PP^3 \,\, \buildrel{{\rm linear}}\over{\hookrightarrow}\,\, \PP^{39}\,\,
\buildrel{{\rm monomial}} \over{\dashrightarrow} \,\, \PP^{39}$.
We choose coordinates
$(c_0:c_1:c_2:c_3)$ on $\PP^3$
and coordinates $(m_0:m_1:\cdots :m_{39})$ on the rightmost $\PP^{39}$.
The $40$ homogeneous coordinates
$u_{ijk\ell}$ on the middle $\PP^{39}$
are indexed by the lines through the origin in the finite vector space
$\mathbb{F}_3^4$. Each line is given
by the vector whose leftmost nonzero coordinate is $1$.
The linear map $\PP^3 \hookrightarrow \PP^{39}$ is defined as follows,
where $\omega = \frac{1}{2} (-1 + \sqrt{-3})$ is a  third root of unity:
$$
\begin{matrix}
u_{0001} = c_1 {+} c_2 {+} c_3 & 
u_{0010} = c_2 {-} c_3 {+} c_0 & 
u_{0011} = c_3 {+} c_0 {-} c_1 & 
u_{0012} = c_0 {+} c_1 {-} c_2 \\
u_{0100} = \sqrt{-3} \cdot c_1 & 
u_{0101} = c_1 {+} \omega^2 c_2 {+} \omega^2 c_3  & 
u_{0102} = c_1 {+} \omega c_2 {+} \omega c_3  &
u_{0110} = c_2 {-} \omega c_3 {+} \omega^2 c_0  \\
u_{0111} = c_3 {+} c_0 {-} \omega c_1  & 
u_{0112} = c_0 {+} \omega^2 c_1 {-} c_2 & 
u_{0120} = c_2 {-} \omega^2 c_3 {+} \omega c_0 & 
u_{0121} = c_0 {+} \omega c_1 {-} c_2 \\
u_{0122} = c_3 {+} c_0 {-} \omega^2 c_1  & 
u_{1000} = \sqrt{-3} \cdot c_0 & 
u_{1001} = c_1 {+} \omega c_2 {+} \omega^2 c_3  & 
u_{1002} = c_1 {+} \omega^2 c_2 {+} \omega c_3 \\
u_{1010} = c_2 {-} c_3 {+} \omega c_0 & 
u_{1011} = c_3 {+} \omega c_0  {-} c_1 & 
u_{1012} = c_0 {+} \omega^2 c_1 {-} \omega^2  c_2  & 
u_{1020} = c_2 {-} c_3 {+} \omega^2  c_0  \\
u_{1021} = c_0 {+} c_1 \omega {-} \omega c_2  & 
u_{1022} = c_3 {+}  \omega^2 c_0 {-} c_1 & 
u_{1100} = \sqrt{-3} \cdot c_3 & 
u_{1101} = c_1 {+} c_2 {+} \omega c_3  \\
u_{1102} = c_1 {+} c_2 {+} c_3 \omega^2 & 
u_{1110} = c_2 {-} \omega c_3 {+} c_0 & 
u_{1111} = c_3 {+} \omega c_0 {-} \omega c_1 & 
u_{1112} = c_0 {+} \omega c_1 {-} \omega^2 c_2 \\
u_{1120} = c_2 {-} \omega^2 c_3 {+} c_0 & 
u_{1121} = c_0 {+} \omega^2 c_1 {-} \omega c_2 & 
u_{1122} = c_3 {+} c_0 \omega^2 {-} \omega^2 c_1 & 
u_{1200} = \sqrt{-3} \cdot c_2 \\
u_{1201} = c_1 {+} \omega^2 c_2 {+} c_3 & 
u_{1202} = c_1 {+} \omega c_2 {+} c_3 & 
u_{1210} = c_2 {-} \omega^2 c_3 {+} \omega^2 c_0 & 
u_{1211} = c_3 {+} \omega c_0 {-} \omega^2  c_1 \\
u_{1212} = c_0 {+} c_1 {-} \omega^2 c_2 & 
u_{1220} = c_2 {-} \omega c_3 {+} \omega c_0 & 
u_{1221} = c_0 {+} c_1 {-} \omega c_2 & 
u_{1222} = c_3 {+} \omega^2 c_0 {-} \omega c_1
\end{matrix}
$$
These $40$ linear forms cut out the hyperplanes
of the complex reflection arrangement $\mathrm{G}_{32}$.
We refer to the book by Hunt \cite[\S 5]{hunt}
for a discussion of this arrangement
and its importance for  modular Siegel threefolds.
Our first map $\PP^3 \hookrightarrow \PP^{39}$ realizes the arrangement $\mathrm{G}_{32}$
as the restriction of the $40$ coordinate planes in
$\PP^{39}$ to a certain  $3$-dimensional linear subspace.

The monomial
 map $\PP^{39} \dashrightarrow \PP^{39}$
is defined outside the hyperplane arrangement 
$ \{\prod u_{ijk\ell} = 0\}$ which corresponds to $\mathrm{G}_{32}$.
It is given  by the following $40$ monomials of degree four:
$$
\begin{matrix}
m_0 = u_{0001} u_{0010} u_{0011} u_{0012} & & 
m_1 = u_{0001} u_{1000} u_{1001} u_{1002} & & 
m_2 = u_{0001} u_{1010} u_{1011} u_{1012}   \\
m_3 = u_{0001} u_{1020} u_{1021} u_{1022} & & 
m_4 = u_{0010} u_{0100} u_{0110} u_{0120} & & 
m_5 = u_{0010} u_{0101} u_{0111} u_{0121}   \\
m_6 = u_{0010} u_{0102} u_{0112} u_{0122} & & 
m_7 = u_{0011} u_{1200} u_{1211} u_{1222} & & 
m_8 = u_{0011} u_{1201} u_{1212} u_{1220}   \\
m_9 = u_{0011} u_{1202} u_{1210} u_{1221} & & 
m_{10} = u_{0012} u_{1100} u_{1112} u_{1121} & &  
m_{11} = u_{0012} u_{1101} u_{1110} u_{1122}  \\
m_{12} = u_{0012} u_{1102} u_{1111} u_{1120} & & 
m_{13} = u_{0100} u_{1000} u_{1100} u_{1200} & &  
m_{14} = u_{0100} u_{1010} u_{1110} u_{1210}  \\
m_{15} = u_{0100} u_{1020} u_{1120} u_{1220} & & 
m_{16} = u_{0101} u_{1000} u_{1101} u_{1202} & &  
m_{17} = u_{0101} u_{1010} u_{1111} u_{1212}  \\
m_{18} = u_{0101} u_{1020} u_{1121} u_{1222} & & 
m_{19} = u_{0102} u_{1000} u_{1102} u_{1201} & &  
m_{20} = u_{0102} u_{1010} u_{1112} u_{1211}  \\
m_{21} = u_{0102} u_{1020} u_{1122} u_{1221} & & 
m_{22} = u_{0110} u_{1001} u_{1111} u_{1221} & &  
m_{23} = u_{0110} u_{1011} u_{1121} u_{1201}  \\
m_{24} = u_{0110} u_{1021} u_{1101} u_{1211} & & 
m_{25} = u_{0111} u_{1001} u_{1112} u_{1220} & &  
m_{26} = u_{0111} u_{1011} u_{1122} u_{1200}  \\
m_{27} = u_{0111} u_{1021} u_{1102} u_{1210} & & 
m_{28} = u_{0112} u_{1001} u_{1110} u_{1222} & &  
m_{29} = u_{0112} u_{1011} u_{1120} u_{1202}  \\
\end{matrix} $$  $$ \begin{matrix}
m_{30} = u_{0112} u_{1021} u_{1100} u_{1212} & & 
m_{31} = u_{0120} u_{1002} u_{1122} u_{1212} & &  
m_{32} = u_{0120} u_{1012} u_{1102} u_{1222}  \\
m_{33} = u_{0120} u_{1022} u_{1112} u_{1202} & & 
m_{34} = u_{0121} u_{1002} u_{1120} u_{1211} & & 
m_{35} = u_{0121} u_{1012} u_{1100} u_{1221}  \\
m_{36} = u_{0121} u_{1022} u_{1110} u_{1201} & & 
m_{37} = u_{0122} u_{1002} u_{1121} u_{1210} & & 
m_{38} = u_{0122} u_{1012} u_{1101} u_{1220}  \\
m_{39} = u_{0122} u_{1022} u_{1111} u_{1200} .
\end{matrix}
$$
The combinatorics behind this list is as follows.
The $40$ monomials represent the $40$ isotropic
planes  in the space $\mathbb{F}_3^4$, with respect to the
symplectic inner product  (\ref{eq:innerproduct}).
The linear inclusion $\PP^3 \hookrightarrow \PP^{39}$
has the property that two linearly independent vectors $x,y$ in $\mathbb{F}_3^4$
satisfy $\langle x,y \rangle = 0$ if and only if
the corresponding linear forms $u_x$ and $u_y$
are perpendicular in the root system $\mathrm{G}_{32}$,
using the usual Hermitian inner product 
(when considered over $\mathbb{C}$).

Let $\mathcal{B}$ denote the Burkhardt quartic in $\PP^{39}$,
that is, the closure of the image of the map  above.
Its homogeneous prime ideal $I_{\mathcal{B}}$
is minimally generated by one quartic and
a $35$-dimensional space of linear forms in
$K[m_0,m_1,\ldots,m_{39}]$.
That space has a natural generating set consisting of $160 = 4 \cdot 40$ linear trinomials.
Namely, the four coordinates  $m_\bullet $ that share a common parameter $u_{ijk\ell}$
span a two-dimensional space modulo $I_\mathcal{B}$. For instance,
the first four coordinates share the parameter $u_{0001}$,
and they satisfy the following linear trinomials:
\begin{equation}
\label{eq:linrel}
\begin{matrix}
 m_0+\omega^2 m_1-\omega m_2 &= &
 m_0-\omega m_1-\omega^2 m_3 &= & & & \\
 m_0+\omega^2 m_2+\omega m_3 &=& 
 m_1+\omega m_2 -\omega^2 m_3 &=& 0 .
 \end{matrix} 
 \end{equation}
 These relations are constructed as follows:
 Each of the $40$ roots $u_{ijk\ell}$ appears as a factor in
 precisely four  of the coordinates $ m_\bullet$, 
 and these four span a two-dimensional space over $K$.

The $160$ linear trinomials (\ref{eq:linrel}) cut out a $4$-dimensional linear subspace of $\mathbb{P}^{39}$. We fix the following system of coordinates, analogous to (\ref{eq:from_m_to_rst}), on that linear subspace $\mathbb{P}^4$ of $\mathbb{P}^{39}$:
\begin{equation}
\label{eq:burkpara}
\begin{matrix}
r & = & 3 c_0 c_1 c_2 c_3 & = & m_{13}/3 \\
s_{01} & = & -c_0(c_1^3+c_2^3+c_3^3) & = & (\sqrt{-3} \cdot m_1 - m_{13})/3 \\
s_{10} & = & \,\,c_1(c_0^3+c_2^3-c_3^3)  & = & (-\sqrt{-3}  \cdot m_4 - m_{13})/3 \\
 s_{11} & = &\,\, c_2(c_0^3-c_1^3+c_3^3) & = & (-\sqrt{-3}\cdot m_7 - m_{13})/3 \\
s_{12} & = &  \,\,c_3(c_0^3+c_1^3-c_2^3) & =  & (-\sqrt{-3} \cdot m_{10} - m_{13})/3 
\end{matrix}
\end{equation}
The polynomial that defines the Burkhardt quartic $\mathcal{B} \subset
 \mathbb{P}^4$ is now written as
 \begin{equation}
\label{eq:burkhardtclassic}
r (r^3+s_{01}^3+s_{10}^3+s_{11}^3+s_{12}^3)\,+\,3 s_{01} s_{10} s_{11} s_{12}
\,\,\, = \,\,\, 0.
\end{equation}
The Burkhardt quartic has $45$ isolated singular points.
For example, one of the singular points is $
(r: s_{01}:s_{10}:s_{11}:s_{12}) = 
(0:0:0:1:1)$. In the $m$-coordinates, this point is
\begin{equation}
\label{eq:singpoint}
\begin{matrix}
(0 : 0 : 0 : 0 : 0 : 0 : 0 : -\omega^2 : -\omega : 1 : \omega^2 : -1 : \omega : 0 : 0 : 0 : 0 : 0 :\\
 0 : 0 : 0 : 0 :
 -\omega^2 :
-1 : -\omega : -1 : -\omega^2 : \omega^2 : -\omega : \omega^2 : 
\\
\omega^2 : \omega^2 : 1 : \omega : 1 : \omega^2 : -\omega^2 : \omega : -\omega^2 : -\omega^2 )
\end{matrix}
\end{equation}
 For each  singular point 
 precisely $16$ of the $40$ $m$-coordinates are zero. 
Each hyperplane $m_\bullet = 0$ intersects the Burkhardt quartic $\mathcal{B}$
in a tetrahedron of four planes, known as {\em Jacobi planes}, 
which contains  $18$ of the $45$ singular points, in a configuration
that is depicted in  \cite[Figure 5.3(b)]{hunt}.
The relevant combinatorics will be explained when tropicalizing in Section~\ref{sec:tropicalization}.

The closure of the image of the monomial map
$\,\PP^{39} \dashrightarrow \PP^{39}, \, u \mapsto m \,$
is a toric variety  $\mathcal{T}$. Writing
$\PP^4$ for the linear subspace
defined by the $160$ trinomials like (\ref{eq:linrel}), we have
\begin{equation}
\label{eq:burkinters}
\mathcal{B} \,\,\, = \,\,\, \mathcal{T} \,\cap \,\PP^4 \quad \subset \quad \PP^{39}. 
\end{equation}
Thus we have realized the Burkhardt quartic as a linear section 
of the toric variety $\mathcal{T}$, and it makes sense to
explore the combinatorial properties of $\mathcal{T}$.
Let $A$ denote the $40 \times 40$ matrix
representing our monomial map $u\mapsto m$.
The columns of $A$ are indexed by the $u_{ijk\ell}$,
and hence by the lines in $\mathbb{F}_3^4$.
The rows of $A$ are indexed by the $m_\bullet$,
and hence by the isotropic planes in $\mathbb{F}_3^4$.
The matrix $A$ is the $0$-$1$ matrix that encodes  incidences
of lines and isotropic planes. Each row and each column
has exactly four entries $1$, and the other entries are $0$.
The matrix $A$ has rank $25$, and we computed its
{\em Markov basis} using the software {\tt 4ti2} \cite{4ti2}.

\begin{proposition} 
\label{prop:toricvarT}
\begin{compactenum}[\rm (a)]
\item The projective toric variety $\mathcal{T}$ has dimension $24$.
\item Its prime ideal is minimally generated by $5136$ binomials, namely
$ 216$ binomials of degree $5$, 
$270$ of degree $6$, 
$4410$ of degree $8$, 
and $240$ of degree $12$.  
\item The Burkhardt quartic is the scheme-theoretic intersection in (\ref{eq:burkinters}).
This intersection is not ideal-theoretic, since there is no quartic relation on $\mathcal{T}$ that 
could specialize to  (\ref{eq:burkhardtclassic}).
\item The $24$-dimensional polytope of $\mathcal{T}$, which
is the convex hull of the $40$ rows of $A$, has
precisely $13144$ facets.
\end{compactenum}
\end{proposition}

\begin{proof}
(a) follows from the fact that ${\rm rank}(A)=25$. The statements in (b) and (c) follow from our {\tt 4ti2} calculation.
 The facets in (d) were computed using the software {\tt polymake} \cite{GJ}.
The scheme-theoretic intersection in (c) can be verified by taking the following five
among the $216$ quintic binomials that vanish on $\mathcal{T}$:
\[
\begin{matrix}
m_0 m_{13} m_{22} m_{33} m_{37} - m_1 m_4 m_9 m_{10} m_{39} & & 
m_0 m_{14} m_{23} m_{33} m_{35} - m_2 m_4 m_9 m_{10} m_{36} \\
m_0 m_{16} m_{25} m_{35} m_{37} - m_1 m_5 m_9 m_{10} m_{38} & &
m_0 m_{17} m_{26} m_{36} m_{38} - m_2 m_5 m_8 m_{11} m_{39} \\
m_9 m_{11} m_{13} m_{18} m_{20} - m_7 m_{10} m_{14} m_{16} m_{21} &  &
\end{matrix}
\]
Each of these quintic binomials factors on $\PP^4$
as the Burkhardt quartic (\ref{eq:burkinters}) times a linear form, and
these five linear forms
generate the irrelevant maximal ideal
$\langle r ,s_{01},s_{10}, s_{11}, s_{12}  \rangle $.
\end{proof}

We next explain the connection to abelian surfaces. Consider
the {\em open Burkhardt quartic} 
\[
\mathcal{B}^{\circ} = \mathcal{B} \setminus \{\prod m_i = 0\} \subset \PP^{39}.
\]
 In its modular interpretation (\cite{FSM}, \cite[\S 3.1]{GS}, \cite[Lemma 5.7.1]{hunt}), this threefold  is  the moduli space
 $\mathcal{M}_2(3)$ of smooth genus $2$ curves with level $3$ structure.
   With every point $(r:s_{01}:s_{10}:s_{11}:s_{12}) \in \mathcal{B}^{\circ}$
 we associate an abelian surface (which is a Jacobian) following \cite[\S 3.2]{GS}.
  The ambient space for this family of abelian surfaces is
  the projective space $\PP^8$ whose coordinates
\[
(x_{00}:x_{01}: x_{02}: x_{10}:x_{11}: x_{12}: x_{20}:x_{21}:x_{22})
\]
are indexed by $\mathbb{F}_3^2$.
The following five polynomials represent all the
affine subspaces of $\mathbb{F}_3^2$:
\begin{align*}
f &\,=\, x_{00}^3+x_{01}^3+x_{02}^3+x_{10}^3+x_{11}^3+x_{12}^3+x_{20}^3+x_{21}^3+ x_{22}^3,\\
g_{01} &\,=\, 3(x_{00}x_{01}x_{02}+x_{10}x_{11}x_{12}+x_{20}x_{21}x_{22}),\\
g_{10} &\,=\, 3(x_{00}x_{10}x_{20}+x_{01}x_{11}x_{21}+x_{02}x_{12}x_{22}),\\
g_{11} &\,=\, 3(x_{00}x_{11}x_{22}+x_{01}x_{12}x_{20}+x_{10}x_{21}x_{02}),\\
g_{12} &\,=\, 3(x_{00}x_{12}x_{21}+x_{01}x_{10}x_{22}+x_{02}x_{11}x_{20}).
\end{align*}
Our abelian surface is the singular locus of the {\em Coble cubic}  $\{C = 0\}$ in $\PP^8$,
which is given~by
 \begin{equation*}
C\,\,=\,\,rf+s_{01}g_{01}+s_{10}g_{10}+s_{11}g_{11}+s_{12}g_{12}.
\end{equation*}

\begin{theorem} \label{thm:abelian93}
The singular locus of the Coble cubic
of any point in $\mathcal{B}^{\circ}$ is an abelian surface $S$
of degree $18$ in $\PP^8$.
This equips $S$ with an indecomposable polarization of type $(3,3)$. 
The prime ideal of $S$ is minimally generated by
$9$ quadrics and $3$ cubics.
The theta divisor on $S$ is a
tricanonical curve of genus $2$, and this is obtained by
intersecting $S$ with the $\mathbb{P}^4$ defined~by
\begin{equation}
\label{eq:ranktrica}
{\rm rank}
\begin{pmatrix}
x_{00} & x_{01} {+} x_{02}  &  x_{10} {+} x_{20} 
& x_{11} {+} x_{22} &  x_{12} {+} x_{21} \\
r & s_{01} & s_{10} & s_{11} & s_{12}
\end{pmatrix}
\,\,\, \leq \,\,\, 1.
\end{equation}
\end{theorem}

\begin{proof}
The first statement is classical (see \cite[\S 10.7]{BL}). We shall explain it below using theta functions.
The fact about ideal generators is due
to Gunji \cite[Theorem 8.3]{Gun}.
The representation (\ref{eq:ranktrica})
of the curve whose Jacobian is $S$ is derived from
\cite[Theorem 3.14(d)]{GS}.
\end{proof}

We now discuss the complex analytic view of our story.
Recall (e.g.~from \cite[\S 8.1]{BL})
 that a principally polarized abelian surface 
over $\mathbb{C}$ is given analytically as
$S_\tau = \mathbb{C}^2/(\mathbb{Z}^2 + \tau \mathbb{Z}^2)$, 
where  $\tau$ is a complex symmetric
$2\times{}2$-matrix whose imaginary part is positive definite.  
The set of such matrices is the {\em Siegel upper half space} $\mathfrak{H}_2$. 
Fix the $4 \times 4$ matrix $J = \begin{bmatrix} 0 \! & \! -\mathrm{Id}_2 \\ \mathrm{Id}_2 \! &
 \! 0 \end{bmatrix}$.
Let $\mathrm{Sp}_4(\mathbb{Z})$ be the group of $4 \times 4$ integer-valued matrices $\gamma$ such that $\gamma J \gamma^T = J$. This acts on $\mathfrak{H}_2$ via
\begin{align} 
\label{eq:actiononh2}
\begin{bmatrix}
A & B\\
C & D
\end{bmatrix}
\cdot{}\tau{} \,\,\,= \,\,(A\tau{}+B)(C\tau{}+D)^{-1},
\end{align}
where $A,B,C,D$ are $2\times 2$ matrices, and this descends to  an action of $\mathrm{PSp}_4(\mathbb{Z})$ on $\mathfrak{H}_2$. The natural map $\mathrm{PSp}_4(\mathbb{Z}) \to \mathrm{PSp}_4(\mathbb{F}_3)$ 
takes the residue class modulo $3$ of each matrix entry. Let $\Gamma_2(3)$ denote the kernel of this map.
The action of $\mathrm{PSp}_4(\mathbb{Z})$ preserves the abelian surface, while $\Gamma_2(3)$ preserves the abelian surface together with a level $3$ structure.  Hence $\mathfrak{H}_2/\mathrm{PSp}_4(\mathbb{Z})$ is the moduli space 
$\mathcal{A}_2$ of principally polarized abelian surfaces, while $\mathfrak{H}_2/\Gamma_2(3)$ is the moduli space
$\mathcal{A}_2(3)$ of principally polarized abelian surfaces with level $3$ structure. 
The finite group $\mathrm{PSp}_4(\mathbb{F}_3)$ is a simple group of order $25920$ and it acts naturally on $\mathfrak{H}_2/\Gamma_2(3)$.

The {\it third-order theta function} with characteristic $\sigma
 \in \frac{1}{3} \mathbb{Z}^2/\mathbb{Z}^2$ is defined as
\begin{align*}
\Theta{}_3[\sigma{}](\tau{},z) &\,=\,\, \,\mathrm{exp}(3\pi{}i\sigma{}^T\tau{}\sigma{}+6\pi{}i\sigma{}^Tz)
\cdot \theta{}(3\tau{},3z+3\tau{}\sigma{}) \\*
&\,=\, \sum_{n\in{}\mathbb{Z}^2}\mathrm{exp}\bigl(3\pi{}i(n{+}\sigma{})^T\tau{}(n{+}\sigma{})
+6\pi{}i(n+\sigma{})^T z \bigr).
\end{align*}
Here $\theta{}$ is the classical Riemann theta function.
For a fixed matrix $\tau{} \in \mathfrak{H}_2$, the nine third-order theta functions on $\C^2$ give precisely our embedding of the abelian surface $S_\tau$ into~$\mathbb{P}^8$:
$$
S_\tau \hookrightarrow \mathbb{P}^8, \qquad
z \mapsto (\Theta{}_3[\sigma](\tau,z))_{\sigma{}\in \frac{1}{3} \Z^2/\Z^2}.
$$
Adopting the notation in \cite[\S 2]{RSSS}, for any $(j,k) \in \{0,1,2\}^2$, we abbreviate
$$
u_{jk} \,=\, \Theta{}_3[(\frac{j}{3},\frac{k}{3})](\tau{},0) \quad \hbox{and} \quad
x_{jk} \,=\, \Theta{}_3[(\frac{j}{3},\frac{k}{3})](\tau{},z).
$$
The nine {\em theta constants} $u_{jk}$ satisfy
$\,u_{01} = u_{02}, \,u_{10} = u_{20}, \, u_{11} = u_{22}$, and
$u_{12}= u_{21}$.
% because
% \begin{align*}
% \Theta{}_3[\sigma{}](\tau{},0) &\,=\, \mathrm{exp}(3\pi{}i\sigma{}^T\tau{}\sigma{})
% \cdot \theta{}(3\tau{},3\tau{}\sigma{})  \\
% &\,=\, \mathrm{exp}(3\pi{}i(-\sigma{})^T\tau{}(-\sigma{})) \cdot \theta{}(3\tau{},3\tau{}(-\sigma{}))
% \,\,=\,\, \Theta{}_3[-\sigma{}](\tau{},0) .
% \end{align*}
For that reason, we need only five theta constants $u_{00},u_{01},u_{10},u_{11},u_{12}$,
which we take as homogeneous coordinates on $\PP^4$. These five
coordinates satisfy one homogeneous equation:

\begin{lemma} \label{lem:Siegel}
The closure of the image of the map $\,\mathfrak{H}_2 \rightarrow \PP^4\,$ given by the
five theta constants is an irreducible hypersurface $\,\mathcal{H}$
of degree $10$. Its defining polynomial is the determinant of 
$$ U \, = \,
\begin{bmatrix}
u_{00}^2 & u_{01}^2      & u_{10}^2       & u_{11}^2        & u_{12}^2\\
u_{01}^2 & u_{00}u_{01} & u_{11}u_{12} & u_{10}u_{12} & u_{10}u_{11}\\
u_{10}^2 & u_{11}u_{12} & u_{00}u_{10} & u_{01}u_{12} & u_{01}u_{11}\\
u_{11}^2 & u_{10}u_{12} & u_{01}u_{12} & u_{00}u_{11} & u_{01}u_{10}\\
u_{12}^2 & u_{10}u_{11} & u_{01}u_{11} & u_{01}u_{10} & u_{00}u_{12}
\end{bmatrix}.
$$
\end{lemma}

\begin{proof}
This determinant appears in \cite[(10)]{DL},  \cite[p.~252]{FSM}, and \cite[\S 2.2]{morikawa}.
\end{proof}

%\steven{Delete following sentence?} A remarkable fact, first noted by Coble, is that
%the Burkhardt quartic is self-Steinerian \cite[\S 5.3]{hunt}. In fact, the
%hypersurface in Lemma \ref{lem:Siegel} is the
%Hessian of the Burkhardt quartic:
%\begin{equation}
%\label{eq:HessBurk}
%  \mathcal{H} \,\, = \,\, {\rm Hess}(\mathcal{B}) .    \end{equation}
At this point, we have left the complex analytic world
and we are back over a more general field $K$.
 The natural map
$\mathcal{H} \dashrightarrow \mathcal{B}$ is
10-to-1 and it is given explicitly by $4 \times 4$-minors of $U$.

\begin{corollary}
\label{cor:coblecubic}
Over the Hessian $\mathcal{H}$ of the Burkhardt quartic, the Coble cubic is written as
\begin{equation}
\label{eq:CobleU}
C \,\,\, = \,\,\, {\rm det}
\begin{bmatrix}
f(\mathbf{x}) & g_{01}(\mathbf{x})      & g_{10}(\mathbf{x})       & g_{11}(\mathbf{x})     & g_{12}(\mathbf{x}) \\
u_{01}^2 & u_{00}u_{01} & u_{11}u_{12} & u_{10}u_{12} & u_{10}u_{11}\\
u_{10}^2 & u_{11}u_{12} & u_{00}u_{10} & u_{01}u_{12} & u_{01}u_{11}\\
u_{11}^2 & u_{10}u_{12} & u_{01}u_{12} & u_{00}u_{11} & u_{01}u_{10}\\
u_{12}^2 & u_{10}u_{11} & u_{01}u_{11} & u_{01}u_{10} & u_{00}u_{12}
\end{bmatrix}.
\end{equation}
For $K=\mathbb{C}$, this expresses 
$r,s_{01},s_{10},s_{11}, s_{12}$ 
as  modular forms in terms of theta constants.
\end{corollary}

We note that the 10-to-1 map $\mathcal{H} \dashrightarrow \mathcal{B}$
is analogous to the 64-to-1 map in \cite[(7.1)]{RSSS}
from the Satake hypersurface onto the
G\"opel variety. The formula for the Coble cubic in
Corollary \ref{cor:coblecubic} is analogous to the expression
 for the Coble quartic in \cite[Theorem 7.1]{RSSS}.

In this section we have now introduced four variants of a universal abelian surface.
Each of these is a five-dimensional projective variety.
Our universal abelian surfaces reside
\begin{compactenum}[\indent (a)]
\item in  $\PP^3 \times \PP^8$ with coordinates $({\bf c}, {\bf x})$,
\item in $\mathcal{B} \times \PP^8 \subset \PP^4 \times \PP^8$ with
coordinates $((r:s_{ij}), {\bf x})$,
\item in
$\mathcal{B} \times \PP^8 \subset \PP^{39} \times \PP^8$ with
coordinates $({\bf m}, {\bf x})$,
\item in $\mathcal{H} \times \PP^8 \subset \PP^4 \times \PP^8$ with
coordinates $({\bf u}, {\bf x})$.
\end{compactenum}
A natural commutative algebra problem is to identify explicit minimal generators for
the bihomogeneous prime ideals of each of these universal abelian surfaces.

For instance, consider case (d).
The ideal contains the polynomial ${\rm det}(U)$
of bidegree $(10,0)$
and eight polynomials of bidegree $(8,2)$,
namely the partial derivatives of $C$
with respect to the $x_{ij}$.
However, these nine do not suffice.
For instance, we have ten linearly independent ideal generators of bidegree
$(3,3)$, namely the $2 \times 2$-minors of the $2 \times 5$-matrix
$$
\begin{bmatrix}
f(\mathbf{x}) & g_{01}(\mathbf{x})      & g_{10}(\mathbf{x})       & g_{11}(\mathbf{x})     & g_{12}(\mathbf{x}) \\
f(\mathbf{u}) & g_{01}(\mathbf{u})      & g_{10}(\mathbf{u})       & g_{11}(\mathbf{u})     & g_{12}(\mathbf{u})
\end{bmatrix}.
$$
These equations have been verified numerically using {\tt Sage} \cite{sage}.
For a fixed  general point ${\bf u} \in \mathcal{S}$, these $2 \times 2$-minors
give Gunji's three cubics that were mentioned in Theorem \ref{thm:abelian93}.

For the case (a) here is a concrete conjecture concerning the desired prime ideal.

\begin{conjecture} \label{conj:ninetythree}
The prime ideal of the universal abelian surface in $\PP^3 \times \PP^8$ is
minimally generated by $93$ polynomials, namely $\,9$
polynomials of bidegree $(4,2)$ and $\,84$  of bidegree $(3,3)$.
\end{conjecture}

The $84$ polynomials of bidegree $(3,3)$ are obtained as the $6 \times 6$-subpfaffians of the matrix
\begin{align} \label{eqn:skew9matrix}
\begin{bmatrix}
0 & -c_0 x_{02} & c_0 x_{01} & -c_1 x_{20} & -c_2 x_{22} & -c_3 x_{21} & c_1 x_{10} & c_3 x_{12} & c_2 x_{11} \\
c_0 x_{02} & 0 & -c_0 x_{00} & -c_3 x_{22} & -c_1 x_{21} & -c_2 x_{20} & c_2 x_{12} & c_1 x_{11} & c_3 x_{10} \\
-c_0 x_{01} & c_0 x_{00} & 0 & -c_2 x_{21} & -c_3 x_{20} & -c_1 x_{22} & c_3 x_{11} & c_2 x_{10} & c_1 x_{12} \\
c_1 x_{20} & c_3 x_{22} & c_2 x_{21} & 0 & -c_0 x_{12} & c_0 x_{11} & -c_1 x_{00} & -c_2 x_{02} & -c_3 x_{01} \\
c_2 x_{22} & c_1 x_{21} & c_3 x_{20} & c_0 x_{12} & 0 & -c_0 x_{10} & -c_3 x_{02} & -c_1 x_{01} & -c_2 x_{00} \\
c_3 x_{21} & c_2 x_{20} & c_1 x_{22} & -c_0 x_{11} & c_0 x_{10} & 0 & -c_2 x_{01} & -c_3 x_{00} & -c_1 x_{02} \\
-c_1 x_{10} & -c_2 x_{12} & -c_3 x_{11} & c_1 x_{00} & c_3 x_{02} & c_2 x_{01} & 0 & -c_0 x_{22} & c_0 x_{21} \\
-c_3 x_{12} & -c_1 x_{11} & -c_2 x_{10} & c_2 x_{02} & c_1 x_{01} & c_3 x_{00} & c_0 x_{22} & 0 & -c_0 x_{20} \\
-c_2 x_{11} & -c_3 x_{10} & -c_1 x_{12} & c_3 x_{01} & c_2 x_{00} & c_1 x_{02} & -c_0 x_{21} & c_0 x_{20} & 0
\end{bmatrix}. 
\end{align}
This skew-symmetric $9 \times 9$-matrix was derived by Gruson and Sam \cite[\S 3.2]{GS},
building on the construction in \cite{GSW},
and it is analogous to the elliptic normal curve in (\ref{eq:pfaff5}).
The nine principal $8 \times 8$-subpfaffians of (\ref{eqn:skew9matrix})
are $x_{00} C, x_{01} C, \ldots, x_{22} C$, where $C$ is the
Coble quartic, now regarded as a polynomial in
$({\bf c}, {\bf x})$  of bidegree $(4,3)$.
Conjecture \ref{conj:ninetythree} is analogous to
\cite[Conjecture 8.1]{RSSS}.
The nine polynomials of bidegree $(4,2)$
are $\partial C/\partial x_{00}, \partial C/\partial x_{01},\ldots,\partial C/\partial x_{22}$.

\smallskip

In the remainder of this section we recall the symmetry groups that act on
our varieties. First there is the complex reflection group 
denoted by $\mathrm{G}_{32}$ in the classification of Shephard and Todd \cite{ST}.
The group $\mathrm{G}_{32}$ is a subgroup  of order $155520$ in  $\mathrm{GL}_4(K)$.
Precisely $80$ of its elements are {\it complex reflections}
 of order $3$. As a linear transformation on $K^4$, each such complex reflection has a triple eigenvalue $1$ and a single eigenvalue $\omega^{\pm 1} = \frac{1}{2} (-1 \pm \sqrt{-3})$.
 
The center of $\mathrm{G}_{32}$ is isomorphic to the cyclic group $\mathbb{Z}/6$. In our coordinates
$c_0,c_1,c_2,c_3$, the elements of the center
are scalar multiplications by $6$th roots of unity. Therefore, this gives an action by 
$\mathrm{G}_{32}/(\mathbb{Z}/6)$ on the  hyperplane arrangement $\mathrm{G}_{32}$ in $\PP^3$.
In fact, we have
\begin{equation}
\label{eq:groupiso}
\frac{\mathrm{G}_{32}}{\mathbb{Z}/6} \,\,
\simeq \,\,\mathrm{PSp}_4(\mathbb{F}_3) .
\end{equation}
The linear map
$\PP^3 \hookrightarrow \PP^{39}$, 
$c \mapsto u$, respects the isomorphism (\ref{eq:groupiso}).
The group acts on the $c$-coordinates by the reflections on $K^4$,
and it permutes the coordinates $u_{ijk\ell}$ via its action on the
lines through the origin in $\mathbb{F}_3^4$. Of course, the
group $\mathrm{PSp}_4(\mathbb{F}_3) $ also permutes the
$40$ isotropic planes in $\PP^{39}$, and this action is compatible with our
 monomial map $\,\PP^{39} \dashrightarrow \PP^{39}$.
 
\section{Tropicalizing the Burkhardt Quartic}
\label{sec:tropicalization}

Our goal is to understand the relationship between
classical and tropical moduli spaces for curves of genus two. To this end,
in this section, we study the tropicalization of the Burkhardt
quartic $\mathcal{B}$. This is a $3$-dimensional fan ${\rm trop}(\mathcal{B})$ in
the tropical projective torus~$\mathbb{TP}^{39}$. We shall see that 
the {\em tropical compactification}  of $\mathcal{B}^{\circ}$
equals the Igusa compactification of~$\mathcal{A}_3(2)$.

The variety $\mathcal{B}$ is the closure of the image of the composition
$\,\PP^3 \hookrightarrow \PP^{39} \dashrightarrow \PP^{39}\,$
of the linear map given by the arrangement $\mathrm{G}_{32}$ and the monomial map given by the $40 {\times} 40$ matrix $A$ that records incidences 
of isotropic planes and lines in $\mathbb{F}_3^4$. To be precise, recall that the source $\PP^{39}$ has coordinates $e_\ell$ indexed by lines $\ell \subset \mathbb{F}_3^4$, the 
target $\PP^{39}$ has coordinates $e_W$ 
indexed by isotropic planes $W \subset \mathbb{F}_3^4$,
 and the linear map $A$ is defined by $A(e_\ell) = \sum_{W \supset \ell} e_W$. 
 This implies the representation
\begin{equation}
\label{eq:tropBerg}
{\rm trop}(\mathcal{B}) \,\, = \,\,
 A \cdot{\rm Berg}(\mathrm{G}_{32}) \quad \subset \quad \mathbb{TP}^{39}
\end{equation}
of our tropical threefold as the image under $A$
of the {\em Bergman fan} of the matroid of $\mathrm{G}_{32}$.
By this we mean the unique coarsest fan structure
on the tropical linear space given by the
rank $4$ matroid on the $40 $ hyperplanes of $\mathrm{G}_{32}$.
This Bergman fan is simplicial, as suggested by
the general theory of \cite{ARW}.
We computed its cones using
the software {\tt TropLi} due to Rinc\'on~\cite{Rin}.

\begin{lemma}
The Bergman complex of the rank $4$ matroid of the complex root system
$\mathrm{G}_{32}$  has  $170$ vertices, $1800$ edges
and $3360$ triangles, so its Euler characteristic equals $1729$.
The rays and cones of the corresponding Bergman fan 
$\,{\rm Berg}(\mathrm{G}_{32}) \subset \mathbb{TP}^{39}$
are described below.
\end{lemma}

The Euler characteristic is the {\em M\"obius number}
of the matroid, which can also be computed 
as the product of the {\em exponents} $n_i$ in \cite[Table 2]{OS}
of the complex reflection group~$\mathrm{G}_{32}$:
$$  1 \cdot 7 \cdot 13 \cdot 19 \,\,=\,\,
1729 \,\, = \,\, 3360 - 1800 + 170 -1 . $$
See \cite[(9.2)]{RSSS} for
the corresponding formula
for the Weyl group of $\mathrm{E}_7$
(and genus $3$ curves).

We now discuss the combinatorics of ${\rm Berg}(\mathrm{G}_{32})$.
The space $\mathbb{TP}^{39} = \R^{40}/\R(1,1,\ldots,1)$
is spanned by unit vectors $e_{0001}, e_{0010}, \ldots, e_{1222}$
that are labeled by the $40$ lines in $\mathbb{F}_3^4$ as before.
The $170$ rays of the Bergman fan correspond to 
the {\em connected flats} of the matroid of $\mathrm{G}_{32}$,
and these come in three symmetry classes,
according to the rank of the connected~flat:
\begin{enumerate}
\item[(a)] $40$ Bergman rays of rank $1$.
These are spanned by the unit vectors $e_{0001},e_{0010},\ldots,e_{1222}$.
\item[(b)] $90$ Bergman rays of rank $2$, such as 
$e_{0001}+e_{0100}+e_{0101}+e_{0102}$,  which
represents
$\,\{ c_1,
c_1{+}c_2{+}c_3,
c_1 {+} \omega c_2 {+} \omega c_3,
c_1{+}\omega^2 c_2{+}\omega^2 c_3\}$.
These are the non-isotropic planes in
$\mathbb{F}_3^4$.
\item[(\"a)] $40$ Bergman rays of rank $3$, such as
$$\,e_{0001}+e_{0010}+e_{0011}+e_{1100}+e_{1101}+e_{1102}
 +e_{1110}+e_{1111}+e_{1112}+e_{1120}+e_{1121}+e_{1122} .$$
These correspond to the Hesse pencils in $\mathrm{G}_{32}$,
and to the hyperplanes in $\mathbb{F}_3^4$.
Note that the $12$ indices above are perpendicular to
$(0,0,1,2)$ in the symplectic inner product.
\end{enumerate}

The $3360$ triangles of the Bergman complex of $\mathrm{G}_{32}$  also come in
three symmetry classes:
\begin{enumerate}
\item[(aa\"a)] Two orthogonal lines (a) together with a hyperplane (\"a) that
contains them both. This gives $480$ triangles because each hyperplane
contains $12$ orthogonal pairs.
\item[(ab\"a)] A flag consisting of a line (a) contained in a non-isotropic plane (b)
contained in a hyperplane (\"a). There are $1440 $ such triangles
since each of the $90$ planes has
$4 \cdot 4$~choices.
\item[(aab)] Two orthogonal lines (a) together with a non-isotropic plane (b).
The plane contains one of the lines and is orthogonal to the other one.
The count is also $1440$.
\end{enumerate}

The $1800$ edges of the Bergman complex come in five symmetry classes:
there are $240$ edges (aa) given by pairs of orthogonal lines,
$360$ edges (ab) given by lines in non-isotropic planes,
$480$ edges (a\"a) given by lines in hyperplanes,
$360$ edges (b\"a) given by non-isotropic planes in hyperplanes,
and $360$ edges (ab${}^\perp$) obtained by dualizing the previous pairs (b\"a).

Our calculations establish the following statement: 

\begin{proposition} \label{prop:CP}
The Bergman complex coincides with the nested set complex for the matroid of $\mathrm{G}_{32}$. In particular, the tropical compactification of the complement of the hyperplane arrangement $\mathrm{G}_{32}$ coincides with the wonderful compactification of de Concini--Procesi \cite{dCP}. 
\end{proposition}

See \cite{feichtner} for the relation between tropical compactifications and wonderful compactifications. 
We expect that Proposition~\ref{prop:CP} is true for any finite complex reflection group, but we have not made any attempts to prove this.

The wonderful compactification is obtained by blowing up the irreducible flats
 of lowest dimension, then blowing up the strict transforms of the irreducible flats of next lowest dimension, etc.  In our case, the smallest irreducible flats  are $40$ points,
 corresponding to the
Bergman rays (a) and to Family 6 in \cite[Table 1]{GS}.
 This first blow-up $\widehat{\PP^3}$ is  the closure of the graph of the  map $\PP^3 \dashrightarrow \mathcal{B}$, by \cite[Proposition 3.25]{GS}.
 The next smallest irreducible flats are the strict transforms of 90 $\PP^1$'s,
 corresponding to the Bergman rays (b) and to Family 4 in \cite[Table 1]{GS}.
 After that, the only remaining irreducible flats are $40$ hyperplanes,
 corresponding to the Bergman rays (\"a) and to Family 2 in \cite[Table 1]{GS}.
   Hence the wonderful compactification $\widetilde{\PP^3}$ is obtained by blowing up these 90 $\PP^1$'s in $\widehat{\PP^3}$. The 90 exceptional divisors of $\widetilde{\PP^3} \to \widehat{\PP^3}$ get contracted to the 45 nodes of $\mathcal{B}$, so we can lift the map $\widetilde{\PP^3} \to \mathcal{B}$ to a map $\widetilde{\PP^3} \to \widetilde{\mathcal{B}}$,
   where $ \widetilde{\mathcal{B}}$
   denotes the blow-up of the Burkhardt quartic at its $45$ singular~points.
   
The hyperplane arrangement complement
 $\PP^3 \cap \mathbb{G}_m^{39}$ is naturally identified with the moduli space $\mathcal{M}_2(3)^-$ of smooth genus 2 curves with level 3 structure and the choice of a Weierstrass point (or equivalently, the choice of an
 odd theta characteristic). See \cite{bolognesi} for more about $\mathcal{M}_2(3)^-$. Hence we have the following commutative diagram
\begin{equation} \label{cd:oddtheta}
\begin{gathered}
\xymatrix{ \mathcal{M}_2(3)^- \ar@{^{(}->}[r] \ar[d] \,&\, 
\widetilde{\PP^3} 
 \ar[d] \ar[r] \,&\, \widehat{\PP^3} \ar[d]  \\
\mathcal{M}_2(3) \ar@{^{(}->}[r] \,&\, \widetilde{\mathcal{B}} \ar[r] \,&\, \mathcal{B} \,&\, } \end{gathered}
\end{equation}
where the vertical maps are generically finite of degree 6, the right
horizontal maps are blow-ups, and the moduli spaces $\mathcal{M}_2(3)^-$ and $\mathcal{M}_2(3)$ are
realized as  very affine varieties.

We now compute the tropical Burkhardt quartic
(\ref{eq:tropBerg}), by applying the linear map $A$
to the Bergman fan of $\mathrm{G}_{32}$. Note that the image lands in
 the tropicalization ${\rm trop}(\mathcal{T})$ of the toric 
variety $\mathcal{T} \subset \PP^{39}$.
We regard ${\rm trop}(\mathcal{T})$ 
as a $24$-dimensional  linear subspace of
$\mathbb{TP}^{39}$.

\begin{theorem} The tropical Burkhardt quartic
${\rm trop}(\mathcal{B})$ is the fan over a $2$-dimensional
simplicial complex with $85$ vertices,
$600$ edges and $880$ triangles.
A census appears in Table~\ref{fig:tropburkorbits}.
 \end{theorem}

\begin{proof}
Given the $\mathrm{G}_{32}$-symmetry, the following properties of the map $A$ can be verified on representatives of $\mathrm{G}_{32}$-orbits.
The linear map $A \colon \mathbb{TP}^{39} \rightarrow \mathbb{TP}^{39}$
has the property that the image of each vector (\"a) equals twice that
of the corresponding unit vector (a). For instance, 
$$ A(e_{0001}{+}e_{0010}{+}e_{0011}{+}e_{1100}{+}e_{1101}
{+}e_{1102}{+}e_{1110}{+}e_{1111}{+}e_{1112}{+}e_{1120}
{+}e_{1121}{+}e_{1122}) \,= \,
2A e_{0012}. $$
Likewise, the $90$ vectors (b) come in natural pairs of
non-isotropic planes that are orthogonal complements.
The corresponding vectors have the same image under $A$. For instance,
\begin{equation}
\label{eq:planepair}
 A(e_{0001}+e_{0100}+e_{0101}+e_{0102}) \,\,=\,\,
A(e_{0010}+e_{1000}+e_{1010}+e_{1020}). 
\end{equation}
We refer to such a pair of orthogonal non-isotropic planes
as a {\em plane pair}. This explains the $85$ rays of
${\rm trop}(\mathcal{B})$, namely, they are the 40 lines $a$ and the
45 plane pairs $\{b,b^\perp\}$ in $\mathbb{F}_3^4$.

The image of each cone of ${\rm Berg}(\mathrm{G}_{32})$ under
$A$ is a simplicial cone of the same dimension. 
There are no non-trivial intersections of  image cones. The map
${\rm Berg}(\mathrm{G}_{32}) \rightarrow {\rm trop}(\mathcal{B})$
is a proper covering of fans. The 2-to-1 covering of the rays
induces a  3-to-1 or 4-to-1 covering  on each higher-dimensional cone.
The precise combinatorics is summarized in Table~\ref{fig:tropburkorbits}.

\begin{table}[h]
\begin{tabular}{|c|c|c|c|c|}
\cline{1-5}
Dimension & Orbits in ${\rm Berg}(\mathrm{G}_{32})$ & The map $A$ & Orbit size in
${\rm trop}(\mathcal{B})$  & Cone type \\
\cline{1-5}
\multirow{3}{*}{$1$} & $40$ (a) & \multirow{2}{*}{$2$ to $1$} & \multirow{2}{*}{$40$} & \multirow{2}{*}{(a)} \\
\cline{2-2}
                                & $40$ (\"a) &                                            &                                   & \\
\cline{2-5}
                                & $90 $ (b) & $2$ to $1$                          & $45$                           & (b) \\
\cline{1-5}
\multirow{5}{*}{$2$} & $240$ (aa) & \multirow{2}{*}{$3$ to $1$} & \multirow{2}{*}{$240$} & \multirow{2}{*}{(aa)}\\
\cline{2-2}
                                & $480$ (a\"a)&                                            &                                   & \\
\cline{2-5}
                                & $360$ (ab) & \multirow{3}{*}{$3$ to $1$} & \multirow{3}{*}{$360$} & \multirow{3}{*}{(ab)}\\
\cline{2-2}
                                & $360$ (b\"a)&                                            &                                   & \\
\cline{2-2}
                                & $360$ (ab${}^\perp$)&                                           &                                   & \\
\cline{1-5}
\multirow{3}{*}{$3$} & $1440$ (aab) & \multirow{2}{*}{$4$ to $1$} & \multirow{2}{*}{$720$} & \multirow{2}{*}{(aab)} \\
\cline{2-2}
                                & $1440$ (ab\"a)&                                            &                                   & \\
\cline{2-5}
                                & $480$ (aa\"a) & $3$ to $1$                          & $160$                           & (aaa) \\ 
\cline{1-5}
\end{tabular}
\label{fig:tropburkorbits}
\caption{Orbits of cones in the tropical Burkhardt quartic}
\smallskip
\end{table}

The types of the cones are named by replacing
\"a with a, and b${}^\perp$ with b.
In total,  there are $40$ vertices of type (a)
and $45$ vertices of type (b).
There are $240$ edges of type (aa),
corresponding to pairs of lines
$a \perp a'$, and $360$ edges of type (ab),
corresponding to inclusions $a \subset b$.
Finally, there are $160$ triangles of type (aaa) and
 $720$ triangles of type (aab).
\end{proof}

\begin{remark} \rm
There is a bijection between the $45$ rays of type (b) in $\mathrm{trop}(\mathcal{B})$ and the $45$ singular points in $\mathcal{B}$. Namely, each vector of type (b) can be written such that $16$ of its coordinates are $1$ and the other coordinates are $0$.
These $16$ coordinates are exactly the $16$ zero coordinates in the corresponding singular point.
Note that the zero coordinates of the particular singular point
in (\ref{eq:singpoint}) form precisely the support of
 the vector (\ref{eq:planepair}). The number $16 $
 arises because each of the $45$ plane pairs $\{b,b^\perp\}$ determines
$16$ of the $40$ isotropic planes: take any
vector in $b$ and any vector in $b^\perp$, and these two will span an isotropic plane. \hfill \qed
\end{remark}

\smallskip

We next consider the {\em tropical compactification}
$\overline{\mathcal{B}}$
of the open Burkhardt quartic $\mathcal{B}^{\circ} = \mathcal{M}_2(3)$.
By definition,  $\overline{\mathcal{B}}$ is the closure
of $\mathcal{B}^{\circ} \subset \mathbb{G}_m^{24}$ inside of the 
toric variety defined by the fan ${\rm trop}(\mathcal{B})$.
This toric variety is smooth because
the rays of the two types of maximal cones (aaa) and (aab)
can be completed to a basis of the lattice 
$\mathbb{Z}^{24}$ spanned by all $85$ rays.

\begin{proposition} \label{prop:tropB-SNC}
The tropical compactification $\overline{\mathcal{B}}$
is sch\"on in the sense of Tevelev \cite{Tev}. 
The boundary $\overline{\mathcal{B}}\setminus \mathcal{B}^{\circ}$
is a normal crossing divisor consisting of $85$
irreducible smooth surfaces.
\end{proposition}

\begin{proof}
Since our fan  on ${\rm trop}(\mathcal{B})$
defines a smooth toric variety, it suffices to show that
all initial varieties $V({\rm in}_v (I_\mathcal{B})) $
are smooth and connected in the torus $\mathbb{G}_m^{39}$ \cite[Lemma 2.7]{hacking}. 
There are six symmetry classes of initial ideals
${\rm in}_v (I_\mathcal{B})$. Each of them is generated by
linear binomials and trinomials together with one non-linear
polynomial $f$, obtained from the quartic by possibly removing monomial factors.
We present representatives for the six classes.
The plane pair $ \{b,b^\perp\}$ appearing in three of the cases
is precisely the one displayed in (\ref{eq:planepair}).
\begin{itemize}
\item[(a)] For the vertex given by the line $(0,0,0,1)$ we take the weight vector
$ v = A e_{0001} = e_0+e_1+e_2+e_3$. Then
                 $f = m_0 m_4^3-3 m_0 m_4 m_7 m_{10}
                    +m_0 m_7^3+m_0 m_{10}^3-3 \sqrt{3}m_1 m_4m_7 m_{10} $.
This bihomogeneous polynomial defines a smooth surface in a quotient torus
$\mathbb{G}_m^3$.
\item[(b)] For the vertex $ \{b,b^\perp\}$ we take the vector (\ref{eq:planepair}).
This is the incidence vector of the zero coordinates in (\ref{eq:singpoint}), namely
 $v =  e_0+e_1+e_2+e_3+e_4+e_5+e_6+e_{13}+
e_{14}+e_{15}+e_{16}+e_{17}+e_{18}+e_{19}+e_{20}+e_{21}$.
The resulting non-linear polynomial is
$ f = - m_0 m_{13}+ m_1 m_4$.
\item[(aa)] For the edge given by the orthogonal lines
$(0,0,0,1)$ and $ (0,0,1,0)$ in $\mathbb{F}_3^4$, we take
$v = 2 e_0+e_1+e_2+e_3+e_4+e_5+e_6 $, and we get
$f = m_0 m_7^3+m_0 m_{10}^3-3 \sqrt{3} m_1 m_4 m_7 m_{10}$.
\item[(ab)] For the edge given by
$(0,0,0,1)$ and $\{b,b^\perp\}$, we take
$v = 2 e_0+2 e_1+2 e_2+2 e_3+e_4+e_5+e_6+e_{13}+e_{14}+e_{15}+e_{16}
+e_{17}+e_{18}+e_{19}+e_{20}+e_{21}$, and we get
$f = - m_0 m_{13}+m_1 m_4$.
\item[(aaa)] For the triangle given by
$(0,0,0,1)$, $(0,0,1,0)$ and $(0,0,1,1)$, we take
$v = 
A(e_{0001} {+}e_{0010} $ $
{+} e_{0011})  = 
3e_0{+}e_1{+}e_2{+}e_3{+}e_4{+}e_5{+}e_6{+}e_7{+}e_8{+}e_9$,
and we get $f =m_0 m_{10}^2- 3 \sqrt{3} m_1 m_4 m_7$.
\item[(aab)] For the triangle given by
$(0,0,0,1)$, $(0,0,1,0)$ and $\{b,b^\perp\}$, we take
$v = 3 e_0{+}2 e_1{+}2 e_2 + 2 e_3{+}2 e_4{+}2 e_5{+}2 e_6{+}e_{13}
{+}e_{14}{+}e_{15}{+}e_{16}{+}e_{17}{+}e_{18}{+}e_{19}{+}e_{20}{+}e_{21}$. Here,
$f {=} - m_{0} m_{13}{+}m_1 m_4$.
\end{itemize}
Note that the polynomials $f$ are the same in the cases (b), (ab) and (aab),
but the varieties $V({\rm in}_v (I_{\mathcal{B}})) $
are different because of the $35$ linear relations.
In cases (b) and (ab) we have both linear trinomials and linear binomials,
while in case (aab) they are all binomials. In all six cases
the hypersurface $\{f=0\}$ has no singular points 
with all coordinates nonzero.
\end{proof}

Our final goal in this section is to equate the tropical compactification $\overline{\mathcal{B}}$ 
with the blown up Burkhardt quartic $\widetilde{\mathcal{B}}$.
By \cite[Theorem 5.7.2]{hunt}, we can identify $\widetilde{\mathcal{B}}$
with the Igusa desingularization of the Baily--Borel--Satake compactification $\mathcal{A}_2(3)^{\rm Sat}$ of $\mathcal{A}_2(3)$. The latter can be constructed as follows.
Let $\widehat{\mathcal{B}}$ be the projective dual variety to $\mathcal{B} \subset \PP^4$.
The canonical birational map $\mathcal{B} \dashrightarrow \widehat{\mathcal{B}}$ is defined outside of the 45 nodes. Since $\mathcal{B}$ is a normal variety, this map factors through the normalization of $\widehat{\mathcal{B}}$, which can be identified with $\mathcal{A}_2(3)^{\rm Sat}$ by \cite[\S 4]{FSM1}. The closure of the graph of the birational map $\mathcal{B} \dashrightarrow \mathcal{A}_2(3)^{\rm Sat}$ is the blow-up of $\mathcal{B}$ at its indeterminacy locus, i.e., the 45 nodes.
 Using \cite[Theorem 3.1]{vdG}, we may identify this with $\widetilde{\mathcal{B}}$. By symmetry, we could also view this as the closure of the image of the inverse birational map. This realizes $\widetilde{\mathcal{B}}$ as a blow-up of $\mathcal{A}_2(3)^{\rm Sat}$, and in particular, the map blows up the Satake boundary $\mathcal{A}_2(3)^{\rm Sat} \backslash \mathcal{A}_2(3)$ which has 40 components all isomorphic to $\mathcal{A}_1(3)^{\rm Sat} \cong \PP^1$.

\begin{lemma} \label{lem:partialtrop-M23}
The moduli space $\mathcal{A}_2(3)$ coincides with 
the partial compactification of $\mathcal{M}_2(3)$ given by the $1$-dimensional subfan of $\mathrm{trop}(\mathcal{B})$ that consists of the $45$ rays of type (b).
\end{lemma}

\begin{proof}
Let $M$ be the partial compactification in question.
Let $\widetilde{\PP^3}$ be the wonderful compactification 
for $\mathrm{G}_{32}$ as described above.
The preimage of the 45 rays of type (b) in ${\rm Berg}(\mathrm{G}_{32})$ consists of 90 rays, and the resulting partial tropical compactification $P$ of $\PP^3 \setminus (40 \text{ hyperplanes})$ is the complement of the strict transforms of the reflection hyperplanes in $\widetilde{\PP^3}$. In the map $P \to \mathcal{B}$, the 90 divisors are contracted to the 45 singular points (2 divisors to each point). We have a map $P \to M$ which maps the 90 divisors of $P$ to the 45 divisors of $M$, and hence the birational map $M \dashrightarrow \mathcal{B}$ (given by the identity map on $\mathcal{M}_2(3)$) extends to a regular map $M \to \mathcal{B}$ which contracts the 45 divisors to the 45 singular points.

 By the universal property of blow-ups, there exists
a map $M \to \widetilde{\mathcal{B}}$ which takes each of the 45 divisors to one of the 45 exceptional divisors of the 
 blow-up $\widetilde{\mathcal{B}} \to \mathcal{B}$.
From our previous discussion, the image of the map $P \to \widetilde{\mathcal{B}}$ equals
 $\mathcal{A}_2(3)$. Since this map factors through $M$, the image of the map $M \to \widetilde{\mathcal{B}}$ is also $\mathcal{A}_2(3)$. This map has finite fibers: this just needs to be checked on the 45 divisors and we can reduce to considering the map $P \to \widetilde{\mathcal{B}}$; in the map $\widehat{\PP^3} \to \widetilde{\mathcal{B}}$, the inverse image of an exceptional divisor is 2 disjoint copies of $\PP^1 \times \PP^1$ and any surjective endomorphism of $\PP^1 \times \PP^1$ has finite fibers. The map is birational and $\mathcal{A}_2(3)$ is smooth, so, by Zariski's Main Theorem \cite[\S III.9]{mumford}, the map is an isomorphism.
\end{proof}

\begin{theorem} \label{thm:IgusaBBS}
The intrinsic torus of $\mathcal{M}_2(3) = \mathcal{B}^{\circ}$ is the dense torus $\mathbb{G}_m^{24}$
of the toric variety $\mathcal{T}$ described in Proposition \ref{prop:toricvarT}.
 The tropical compactification of $\mathcal{M}_2(3)$ provided by $\mathrm{trop}(\mathcal{B})$ is the Igusa desingularization $\widetilde{\mathcal{B}}$ of the Baily--Borel--Satake compactification $\mathcal{A}_2(3)^{\rm Sat}$ of $\mathcal{A}_2(3)$.
\end{theorem}

\begin{proof}
The first statement follows from Lemma~\ref{lem:intrinsictorusimage} and Proposition~\ref{prop:toricvarT}(a).

By Lemma~\ref{lem:partialtrop-M23}, $\overline{\mathcal{B}}$ is a compactification of $\mathcal{A}_2(3)$. The boundary of the compactification $\mathcal{M}_2(3) \subset \overline{\mathcal{B}}$ is a normal crossings divisor (Proposition~\ref{prop:tropB-SNC}), so the same is true for the boundary of  $\mathcal{A}_2(3) \subset \overline{\mathcal{B}}$, and hence it is toroidal. So there exists a map $f \colon \overline{\mathcal{B}} \to \mathcal{A}_2(3)^{\rm Sat}$ that is the identity on $\mathcal{A}_2(3)$ \cite[Proposition III.15.4(3)]{borelji}. This map is unique and surjective.

From what we said above, the Satake boundary $\mathcal{A}_2(3)^{\rm Sat} \backslash \mathcal{A}_2(3)$ has 
$40$ components all isomorphic to $\PP^1$. Also, $\overline{\mathcal{B}} \backslash \mathcal{A}_2(3)$ consists of 
$40$ divisors. Hence the map $f$ contracts the
$40$ divisors to these $\PP^1$'s. By the universal property of blow-ups, there is a unique map $\tilde{f} \colon \overline{\mathcal{B}} \to \widetilde{\mathcal{B}}$ that commutes with the blow-up map. Then $\tilde{f}$ is birational and surjective. We know this map is an isomorphism on $\mathcal{A}_2(3)$ and the complement of this open subset in both domain and target are a union of $\PP^2$'s. Any surjective endomorphism of $\PP^2$ has finite fibers, and hence $\tilde{f}$ is an isomorphism by Zariski's Main Theorem \cite[\S III.9]{mumford} since $\widetilde{\mathcal{B}}$ is smooth.
\end{proof}

\section{Moduli of Genus Two Curves}
\label{sec:genus2moduli}

The moduli space $\mathcal{M}_{g,n}^{\rm tr}$ of tropical curves of genus $g$ with $n$ marked points is a {\it stacky fan}. This was shown by
Brannetti, Melo and Viviani \cite{BMV} and Chan \cite{Cha}.
This space was studied by many authors. See \cite{caporaso1, caporaso2} for some results.
Here, a {\it tropical curve} is a triple $(\Gamma,w,\ell)$, 
where $\Gamma =(V,E)$ is a connected graph, 
$w$ is a weight function $V\to \mathbb{Z}_{\geq{}0}$, 
and $\ell$ is a length function $E\to \mathbb{R}_{\geq{}0}$. 
The {\it genus} of a tropical curve is the sum of weights of all vertices plus the genus of the graph $\Gamma{}$. 
In addition to identifications induced by graph automorphisms, 
two tropical curves are isomorphic if one can be obtained from another 
by a sequence of the following operations and their inverses:
\begin{compactitem}
\item Removing a leaf of weight $0$, together with the only edge connected to it.
\item Removing a vertex of degree $2$ of weight $0$, and replacing the two edges connected to it with an edge whose length is the sum of the two old edges.
\item Removing an edge of length $0$, and combining the two vertices connected by that edge. The weight of the new vertex is the sum of the two old vertices.
\end{compactitem}
In this way, every tropical curve of genus $\geq{}2$ is uniquely represented by a {\it minimal skeleton}, 
i.e., a tropical curve with no vertices of weight $0$ of degree $\leq{}2$ or edges of length $0$. 
The moduli space of tropical curves with a fixed {\em combinatorial type} $(\Gamma{},w)$ is $\mathbb{R}_{>0}^{|E|}/\mathrm{Aut}(\Gamma{})$, 
where the coordinates of $\mathbb{R}_{>0}^{|E|}$ are the lengths of the edges. 
The cones for all combinatorial types are glued together to form $\mathcal{M}_g^{\rm tr}$.
The boundary of the cone of a combinatorial type $(\Gamma,w)$
corresponds to tropical curves with at least one edge of length $0$.
Contracting that edge gives a combinatorial type $(\Gamma',w')$. 
Then, the cone for $(\Gamma{}',w')$ is glued along the boundary of the cone for $(\Gamma{},w)$ in the natural way.
More generally, a tropical curve with {\it marked points} is defined similarly, 
but allowing rays connecting a vertex with leaves ``at infinity''.

The following construction maps curves over a valued field to tropical curves.
It is fundamental for \cite{ACP, BPR}.
Our description follows  \cite[Lemma - Definition 2.2.4]{viviani}.
Let $R$ be a complete discrete valuation ring with maximal ideal $\mathfrak{m}$. Let $K$ be its field of fractions, $k = R / \mathfrak{m}$ its residue field, and  $t \in R$ a uniformizing parameter.
Fix a genus $g$ curve $C$ with $n$ marked points over $K$.
The curve $C$ is a morphism ${\rm Spec}\ K \to \mathcal{M}_{g,n}$. Since the stack $\overline{\mathcal{M}}_{g,n}$ is proper (i.e., by the stable reduction theorem), there is a finite extension $K'$ of $K$ with discrete valuation ring $R'$ such that this morphism extends uniquely to a morphism ${\rm Spec}\ R' \to \overline{\mathcal{M}}_{g,n}$ (we call this a {\it stable model} of $C$). Here we renormalize the valuation on $R'$ so that its value group is $\mathbb{Z}$. Reducing modulo $\mathfrak{m}'$ gives us a point ${\rm Spec}\ k \to \overline{\mathcal{M}}_{g,n}$. By definition, this is a {\em stable curve} $C_k$ over $k$. We remark that the stable model may not be unique, but the stable curve is unique. 
Since such a stable curve has at worst nodal singularities,  we can construct a dual graph as follows. 
For each genus $h$ component of $C_k$, we draw a vertex of weight $h$.
For each node of $C_k$, we draw an edge between the two components that meet there (this might be a loop if the node comes from a self-intersection). If a component has a marked point, then we attach a vertex at infinity to that vertex. The stable condition translates to the fact that the dual graph is a minimal skeleton as above. Finally, each node, when considered as a point in $C_{R'}$, is \'etale locally of the form $xy = t^\ell$ for some positive integer $\ell$. We then assign the length $\ell / d$ to the corresponding edge, where $d$ is the degree of the field extension $K \subset K'$. In this way, we have defined a function 
\begin{equation}
\label{eq:classicaltotropical}
\mathcal{M}_{g,n}(K) \to \mathcal{M}_{g,n}^{\rm tr}.
\end{equation}
Here is a concrete illustration of this function for $g = 0$ and $n=4$.

\begin{example} \label{ex:stevens4points} \rm
Let $K = \C(\!(t)\!)$ and $R = \C[\![t]\!]$ and consider the four points in $\PP^1_K$ given by 
$(1\!:\!p(t)), (1\!:\!q(t)), (1\!:\!a), (1\!:\!b)$ where $a, b \in \C$ are generic and ${\rm val}(p(t)), {\rm val}(q(t)) > 0$. Let $x,y$ be the coordinates on $\PP^1$. Naively, this gives us four points in $\PP^1_R$, but it is not a stable model since $p(0) = q(0) = 0$ and so two points coincide in the special fiber. 
The fix is to blow-up the arithmetic surface $\PP^1_R$ at the ideal $\langle y-p(t)x, y-q(t)x \rangle$. We embed this blow-up into $\PP^1_R \times_R \PP^1_R$, where the latter $\PP^1_R$ has coordinates $w,z$, as the hypersurface given by $w(y-q(t)x) = z(y-p(t)x)$. The special fiber is the nodal curve given by $y(w-z) = 0$. We wish to understand the \'etale local equation for the node cut out by $y=w-z=0$. To do this, set $x=z=1$ and consider the defining equation $y(w-1) + p(t)-q(t)w=0$. Now substitute $w' = w-1$ and $y' = y-q(t)$ to get $y'w' + (p(t)-q(t)) = 0$. Hence the dual curve is a line segment of length  ${\rm val}(p(t)-q(t))$
with both vertices having weight $0$.
\qed
\end{example}

Evaluating the map (\ref{eq:classicaltotropical}) in general is a challenging 
computer algebra problem:
how does one compute the metric graph from a smooth curve $C$ that
is given by explicit polynomial equations over $K$?
This section represents a contribution to this problem for curves of genus $2$.
As a warm-up for our study of genus 2 curves, let us first consider the genus 1 case.
 
\begin{example} \label{ex:bernds4points} \rm
An elliptic curve $C$ can be defined by giving
four points in $\PP^1$. The curve is the double cover
of $\PP^1$ branched at those four points.
This gives us a map $\mathcal{M}_{0,4} \rightarrow \mathcal{M}_1$,
which is well-defined over our field $K$.
The map is given explicitly
by the following formula for the j-invariant of  $C$ in terms of the cross ratio
 $\lambda{}$ of four ramification points (see \cite[\S 3]{Tev2}):
\begin{equation}
\label{eq:jinv}
j \,\,=\,\, 256\frac{(\lambda{}^2-\lambda{}+1)^3}{\lambda{}^2(\lambda{}-1)^2}.
\end{equation}
We now pass to the tropicalization by constructing a commutative square
\begin{equation} \label{eqn:jinv-comm}
\begin{diagram}
\mathcal{M}_{0,4}(K) & \rTo  & \mathcal{M}_{0,4}^{\rm tr} \\
\dTo & & \dTo \\
\mathcal{M}_1(K) & \rTo & \mathcal{M}_1^{\rm tr}
\end{diagram}
\end{equation}
The horizontal maps are instances of
(\ref{eq:classicaltotropical}), and
the left vertical map is (\ref{eq:jinv}).
Our task is to define the right vertical map.
The ingredients are
the trees and tropical curves in Table~\ref{table:genus1}:

\begin{table}[h!]
\centering
\begin{tabular}{|c|c|}
\cline{1-2}
Tropical curve of genus $1$ & Tree with $4$ leaves\\
\cline{1-2}
\includegraphics{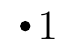} & \includegraphics{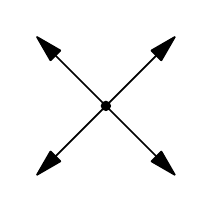} \\
\cline{1-2}
\includegraphics{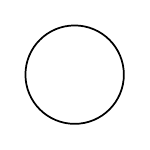} & \includegraphics{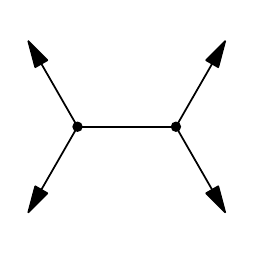} \\
\cline{1-2}
\end{tabular}
\label{table:genus1}
\caption{Trees on four taxa and tropical curves of genus 1}
\end{table}

A point in $\mathcal{M}_{0,4}^{\rm tr}$ can be represented by
 a phylogenetic tree
with taxa $1,2,3,4$.   Writing $\nu_{ij}$ for half the
distance from leaf $i$ to leaf $j$ in that tree, the unique
 interior edge has length
$$ \ell \,\, = \,\, 
{\rm max} \bigl\{
\nu_{12}+\nu_{34} - \nu_{14} - \nu_{23},\,
  \nu_{13}+\nu_{24} - \nu_{12} - \nu_{34},\,
\nu_{14}+\nu_{23} - \nu_{13} - \nu_{24} \bigr\}. $$
Suppose we represent a point in $\mathcal{M}_{0,4}$
by four scalars, $x_1,x_2,x_3, x_4 \in K$, as in Example~\ref{ex:stevens4points}. 
Then its image in
$\mathcal{M}_{0,4}^{\rm tr}$ is the phylogenetic tree obtained by setting
\begin{equation}
\label{eq:nu}  \nu_{ij}  = - {\rm val} (x_i - x_j) . 
\end{equation}

The square \eqref{eqn:jinv-comm} becomes commutative if the right vertical map takes trees with interior edge length $\ell > 0$ to the
  cycle of length $2 \ell$, and it takes
the star tree  $(\ell = 0)$
to the node marked $1$.
To see this, we recall that
the tropical curve contains a cycle
of length $- {\rm val}(j)$, where $j$ is the j-invariant.
This is a standard fact (see \cite[\S 7]{BPR})
from the theory of elliptic curves over $K$.
Suppose the four given points in $\PP^1$
  are $0,1, \infty,\lambda$, and that $\lambda$ and $0$ are neighbors in the tree. This means $\mathrm{val}(\lambda)>0$. As desired, the length of our cycle is
  \begin{align*}
-\mathrm{val}(j) &= -3\mathrm{val}(\lambda^2-\lambda+1)+2\mathrm{val}(\lambda{})+2\mathrm{val}(\lambda{}-1) \,=\, 0+2\mathrm{val}(\lambda{})+0\,=\,
2 \mathrm{val}(\lambda{}).
\end{align*}
The other case, when $\lambda $ and $0$ are not neighbors in the tree,
follows from the fact that the rational function of $\lambda$ in (\ref{eq:jinv})
is invariant under permuting the four ramification points. 

In this example, the one-dimensional fan $\mathcal{M}_{0,4}^{\rm tr}$ serves
as a moduli space for tropical elliptic curves. A variant where the fibers
are elliptic normal curves is shown in Figure~\ref{fig:pentagonpentagram}.
In both situations, all maximal cones
correspond to elliptic curves over $K$ with bad reduction. 
\qed
\end{example}

Moving on to genus $2$ curves, we shall now focus on 
the tropical spaces $\mathcal{M}_2^{\rm tr}$ and $\mathcal{M}_{0,6}^{\rm tr}$. 
There are seven combinatorial types for genus $2$ tropical curves.
Their poset is shown in \cite[Figure 4]{Cha}.
The seven types are drawn in the second column of Table~\ref{table-tropical-moduli}.
The stacky fan $\mathcal{M}_2^{\rm tr}$ is the cone over the
two-dimensional cell complex shown in Figure~\ref{figure-tropical-moduli}. 
Note the identifications.

\begin{figure}[h]
\begin{center}
\includegraphics[width=0.35\textwidth]{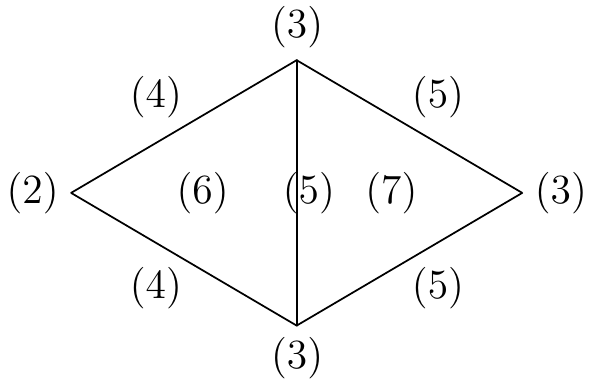}
\end{center}
\vspace{-0.3in}
\label{figure-tropical-moduli}
\caption{The moduli space of genus $2$ tropical curves}
\end{figure}

The tropical moduli space $\mathcal{M}_{0,6}^{\rm tr}$ is the space
of phylogenetic trees on six taxa.
A concrete model, embedded in $\mathbb{TP}^{14}$,
is the $3$-dimensional fan ${\rm trop}(\mathcal{M}_{0,6})$
seen in Theorem \ref{thm:tropsegrecubic}.
Combinatorially, it agrees with the tropical Grassmannian ${\rm Gr}(2,6)$ 
as described in \cite[Example 4.1]{SS}, so its cones
correspond to trees with six leaves.
The fan $\mathcal{M}_{0,6}^{\rm tr}$ has one zero-dimensional cone of type (1),
$25 = 10+15 $ rays of types (2) and (3),
$105 = 60+45$ two-dimensional cones of types (4) and (5),
and $105 = 90 + 15$ three-dimensional cones of types (6) and (7).
The corresponding combinatorial types of  trees are 
depicted in the last column of Table~\ref{table-tropical-moduli}.

Table~\ref{table-tropical-moduli} shows that there is a combinatorial correspondence
between the types of cones of the tropical Burkhardt quartic
${\rm trop}(\mathcal{B}) $ in Table~\ref{fig:tropburkorbits} and the types of cones in
$\mathcal{M}_2^{\rm tr}$ and $\mathcal{M}_{0,6}^{\rm tr}$.
We seek to give a precise explanation of this correspondence
in terms of algebraic geometry. At the moment we can carry
this out for level $2$ but we do not yet have a proof for level $3$.

\begin{table}
\centering
\begin{tabular}{|c|c|c|c|}
\cline{1-4}
Label & Tropical curve of genus $2$ & Burkhardt cone & Tree with $6$ leaves \\
\cline{1-4}
(1) & \includegraphics{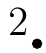} & origin & \includegraphics{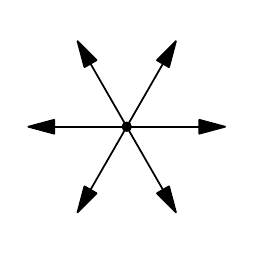} \\
\cline{1-4}
(2) & \includegraphics{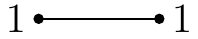} & (b) & \includegraphics{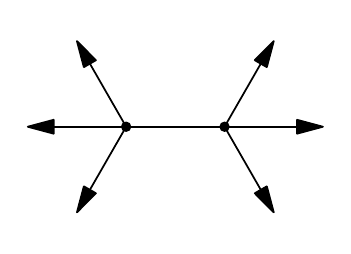} \\
\cline{1-4}
(3) & \includegraphics{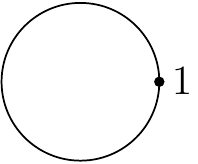} & (a) & \includegraphics{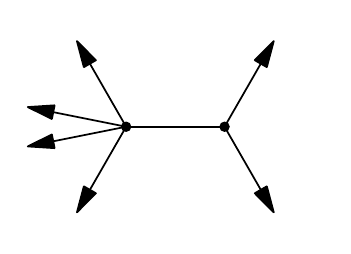} \\
\cline{1-4}
(4) & \includegraphics{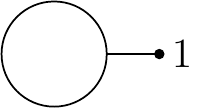} & (ab) & \includegraphics{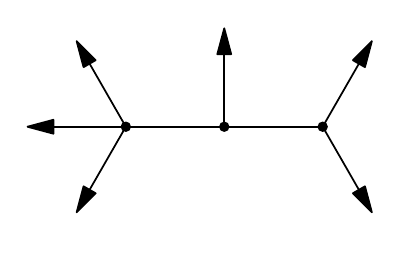} \\
\cline{1-4}
(5) & \includegraphics{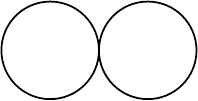} & (aa) & \includegraphics{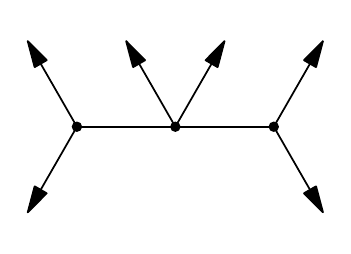} \\
\cline{1-4}
(6) & \includegraphics{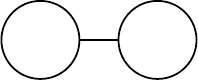} & (aab) & \includegraphics{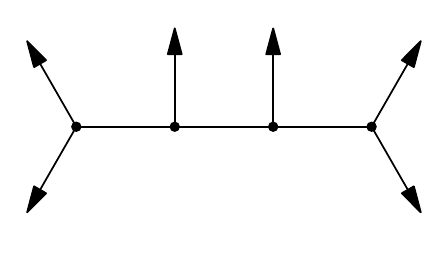} \\
\cline{1-4}
(7) & \includegraphics{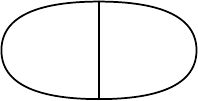} & (aaa) & \includegraphics{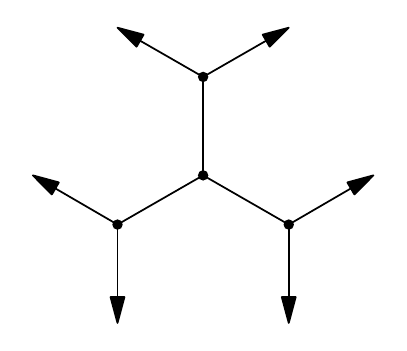} \\
\cline{1-4}
\end{tabular}
\label{table-tropical-moduli}
\caption{Correspondence between tropical curves,  cones of ${\rm trop}(\mathcal{B})$,
and metric trees.}
\end{table}

\begin{theorem} \label{thm:level2trop}
Let $K$ be a complete nonarchimedean field. 
\begin{compactenum}[\rm (a)]
\item There is a commutative square
\begin{equation}
\begin{diagram}
\mathcal{M}_{0,6}(K) & \rTo  & \mathcal{M}_{0,6}^{\rm tr} \\
\dTo & & \dTo \\
\mathcal{M}_2(K) & \rTo & \mathcal{M}_2^{\rm tr}
\end{diagram}
\end{equation}
The left vertical map sends $6$ points in $\mathbb{P}^1$ to the genus $2$ hyperelliptic curve with these ramification points. The horizontal maps send a curve (with or without marked points) to its tropical curve (with or without leaves at infinity). The right vertical map is a morphism of generalized cone complexes 
relating the second and fourth columns of Table~\ref{table-tropical-moduli}.
\item The top horizontal map can be described in an alternative way: under the embedding of
$\mathcal{M}_{0,6} $ into $ \PP^{14}$ given by \eqref{eq:segremap}, take the valuations of the $15$ coordinates $m_0,m_1,\ldots,m_{14}$.
\end{compactenum}
\end{theorem}

\begin{proof}[Proof of Theorem~\ref{thm:level2trop}]
We start with (a). Let $C$ be a genus $2$ curve over $K$ and let $p_1, \dots, p_6 \in \PP^1_K$ be the 
 branch points of the double cover $C \to \PP^1$ 
induced from the canonical divisor. Let $R'$ be a discrete valuation ring over which a stable model of both $C$ and $(\PP^1, p_1, \dots, p_6)$ can be defined and let $k$ be its residue field. The fact that the combinatorial types of dual graphs for $C$ and the marked curve $(\PP^1, p_1, \dots, p_6)$ match up as in Table~\ref{table-tropical-moduli} is clear from the proof of \cite[Corollary 2.5]{avritzer} which constructs the stable $k$-curve of $C$ from that of $(\PP^1, p_1, \dots, p_6)$. There is an obvious bijection of edges between the combinatorial types in all cases. We claim that the edge length coming from the \'etale neighborhood of nodal singularities is halved for curves of type (2) and is doubled for curves of type (3) from Table~\ref{table-tropical-moduli}: the description and proof for the other types can be reduced to these two cases.

First consider curves of type (3). Our stable genus $0$ curve consists of the union of two $\PP^1$'s meeting in a point. One has $4$ marked points and the other has $2$ marked points. This arises from $6$ distinct points in $\PP^1_{R'}$ such that exactly $2$ of them coincide after passing to the residue field. To build a stable model (cf. Example~\ref{ex:stevens4points}), we blow up the point of intersection in the special fiber of $\PP^1_{R'}$ to get an arithmetic surface $\tilde{P}_{R'}$. Let $E$ be the double cover of the first $\PP^1_k$ along the $4$ marked points, and let $E'$ be a copy of $\PP^1_k$ mapping to the second $\PP^1_k$ so that it is ramified over the $2$ marked points. Over the point of intersection, both $E$ and $E'$ are unramified, and we glue together the two preimages (there are two ways to do this, but the choice won't matter). Then $E \cup E'$ is a semistable (but not stable) curve which is the special fiber of an admissible double cover $C_{R'} \to \tilde{P}_{R'}$. Suppose that the node in the special fiber of $\tilde{P}_{R'}$ \'etale locally is $xy = t^\ell$. In a small neighborhood of this node, there are no ramification points.
Thus, a small neighborhood of each of these two points of intersection in $C_{R'}$ is isomorphic to a small neighborhood of the node in $\tilde{P}_{R'}$ and hence \'etale locally look like $xy = t^{\ell}$. Finally, we have to contract $E'$ to a single point to get a stable curve. The result is that the two nodes become one which \'etale locally looks like $xy = t^{2\ell}$.

Now consider curves of type (2). Use the notation from the previous case. The semistable model $C_{R'}$ over ${\rm Spec}\ R'$ has a hyperelliptic involution whose quotient is the union of two $\PP^1_{R'}$'s. At the node of $C_{R'}$, which locally looks like $R'[\![x,y]\!]/\langle xy-t^m \rangle$ 
for some $m$, the involution negates $x$ and $y$ since it preserves the two components of $C_{R'}$. The ring of invariants
 is $R'[\![u,v]\!]/\langle uv-t^{2m} \rangle $ where $u{=}x^2$, $v{=}y^2$. This is the local picture for the nodal genus $0$~curve. 

The result above can also be deduced from Caporaso's
general theory in \cite[\S 2]{Cap}.
For a combinatorial illustration of type (6) see Chan's Figure 1 in \cite{Cha2}.
The two leftmost and two rightmost edges in her upstairs graph have been contracted away.
What is left is a ``barbell" graph with five horizontal edges of  lengths $a, a, b, c, c$,   mapping harmonically to a downstairs graph of edge lengths $a, 2b, c$.  
Here we see both of the stretching factors represented in different parts of this harmonic morphism: 
a $2$-edge cycle of total length $a+a$ maps to an edge of length $a$, and a single edge of length $b$ maps to an edge of length
$ 2b$ downstairs. 

\smallskip

Now we consider (b). We need to argue that the
internal edge lengths can be computed from the $15$ quantities
${\rm val}(m_i)$, in a manner that is consistent with the description above.
For genus 1 curves this is precisely the consistency between Examples~\ref{ex:stevens4points}
and \ref{ex:bernds4points}.
We explain this for the case of the snowflake tree (7).
Without loss of generality, we assume that
$\{1,2\}$, $\{3,4\}$ and $\{5,6\}$ are the
neighbors on the tree. If  $\nu_{ij}$ is
half the distance between leaves $i$ and $j$,
computed from the six points as in (\ref{eq:nu}), then, for instance,
$\,{\rm val}(m_{13}) =  -\nu_{16} - \nu_{24} - \nu_{35} $.
A direct calculation on the snowflake tree shows that the
three internal edge lengths are
\begin{equation}
\label{eq:internaledges}
{\rm val}( m_{2})-{\rm val}(m_{13})  ,\,
{\rm val}(m_{6})-{\rm val}(m_{13}), \,\,\hbox{and} \,\,\,
{\rm val}(m_{14})- {\rm val}(m_{13}) . 
\end{equation}
The edge lengths of the tropical curve $\,$
\includegraphics[width=0.05\textwidth]{figure7} $\,$
are gotten by doubling these numbers.
\end{proof}

At present we do not know the level $3$ analogues to the stretching factors $1/2$ and $2$ we saw in the proof above. Such lattice issues will play a role for the natural map from the tropical Burkhardt quartic onto the tropical moduli space $\mathcal{M}_2^{\rm tr}$. We leave that for future research:

\begin{conjecture}
\label{conj:tropmod}
Let $K$ be a complete nonarchimedean field. There is a commutative square
\begin{equation}
\begin{diagram}
\mathcal{M}_2(3)(K) & \rTo  & \mathrm{trop}(\mathcal{B}) \\
\dTo                       &        & \dTo \\
\mathcal{M}_2(K)    & \rTo & \mathcal{M}_2^{\rm tr}
\end{diagram}
\end{equation}
The left map is the forgetful map. The top map is taking valuations of the coordinates $\,m_0,\dotsc{},m_{39}$. The bottom map sends a curve to its tropical curve. The right map is a morphism of (stacky) fans that takes the third column of Table~\ref{table-tropical-moduli} to the second column.
\end{conjecture}

Here is one concrete way to evaluate the left vertical map $\mathcal{M}_2(3) \to \mathcal{M}_2$
over a field $K$. We can represent an element of $\mathcal{M}_2(3)$ by
a point  $(r:s_{01}:s_{10}:s_{11}:s_{12}) \in \PP^4_K$
that lies in the open Burkhardt quartic
$\mathcal{B}^{\circ}$. The corresponding abelian surface $S$ is
the singular locus of the Coble cubic in $\PP^8_K$ by Theorem 
\ref{thm:abelian93}. If we intersect the abelian surface $S$ with the
linear subspace $\PP^4$ given by (\ref{eq:ranktrica}), then
the result is the desired genus $2$ curve $C \in \mathcal{M}_2(K)$.
The conjecture asks about the precise relationship between 
  the tropical curve constructed from $C$ 
  and  the valuations of our $40$ 
  canonical coordinates $m_0,m_1, \ldots, m_{39}$ 
  on $\mathcal{B}^{\circ} $ inside $ \PP^{39}_K$.

\section{Marked Del Pezzo surfaces}
\label{sec:delpezzo}

This section is motivated by our desire to draw all combinatorial types of tropical cubic surfaces together with their $27$ lines (trees). These surfaces arise in fibers of the map from a six-dimensional fan to a four-dimensional fan. These tropical moduli spaces were characterized by Hacking, Keel and Tevelev in \cite{HKT}. We now rederive their fans from first principles.

Consider a reflection arrangement of type $\mathrm{E}_n$ for $n=6,7$. The complement of the hyperplanes is the moduli space of $n$ points in $\PP^2$ in general position (no 2 coincide, no 3 are collinear, no 6 lie on a conic)
 together with a cuspidal cubic through these points (none of which is the cusp). For $n=6$, there is a $1$-dimensional family of such curves (this family is the {\it parabolic curve} in \cite[Definition 3.2]{CGL}).
 For $n=7$ there are $24$ choices.  We can use maps (\ref{eq:twomaps}) that come from Macdonald representations to forget the data of the cuspidal~cubic.

Consider the case $n=6$. Six points on a cuspidal cubic in $\PP^2$ 
 are represented by a matrix
\begin{equation}
\label{eq:threebysix}
D \,\,\,= \,\,\,
\begin{pmatrix} 1 & 1 & 1 & 1 & 1 & 1 \\ d_1 & d_2 & d_3 & d_4 & d_5 & d_6  \\\
 d_1^3 & d_2^3 & d_3^3 & d_4^3 & d_5^3 & d_6^3 \end{pmatrix}.
\end{equation}
The maximal minors of this $3 \times 6$-matrix are denoted
\[
[ijk] \quad = \quad (d_i-d_j)(d_i-d_k)(d_j-d_k) (d_i+d_j+d_k) 
\qquad {\rm for} \quad 1 \leq i < j < k \leq 6.
\]
We also abbreviate the condition for the six points to lie on a conic:
\[
[{\rm conic}] \, = \,
[134][156][235][246] -  [135][146][234][256] \,\,= \,\,
(d_1+d_2+d_3+d_4+d_5+d_6) \!\! \prod_{1 \leq i < j \leq 6} \!\!\!\! (d_i-d_j) .
\]
The reflection arrangement of type $\mathrm{E}_6$ consists of the
 $36 = \binom{6}{3}+\binom{6}{2}+1$ hyperplanes defined by
 the linear forms in the products above.
We list the flats of this  arrangement  in Table~\ref{fig:e6flats}. 
The bold numbers indicate irreducible flats.
Each flat corresponds to a root subsystem, but not conversely.
Root subsystems that are not parabolic, such as
$\mathrm{A}_2^{\times 3}$, do not come from flats.

\begin{table}[h]
\small
\begin{tabular}{l|l|l|l|l}
\# & Codim & Size & Root subsystem & Equations of a representative flat \\
\hline
{\bf 1} & 1 & 36 & $\mathrm{A}_1$ & $d_1-d_2$\\
\hline
{\bf 2} & 2 & 120 & $\mathrm{A}_2$ & $d_1+d_3+d_6, d_2+d_4+d_5$\\
3 & 2 & 270 & $\mathrm{A}_1 \times \mathrm{A}_1$ & $d_1+d_4+d_6, d_2+d_4+d_5$\\
\hline
{\bf 4} & 3 & 270 & $\mathrm{A}_3$ & $d_1+d_4+d_6, d_2+d_4+d_5, d_5-d_6$\\
5 & 3 & 720 & $\mathrm{A}_2 \times \mathrm{A}_1$ & $d_1+d_5+d_6, d_2+d_4+d_5, d_4-d_5$\\
6 & 3 & 540 & $\mathrm{A}_1^{\times 3}$ & $d_1+d_4+d_6, d_2+d_3+d_6, d_2+d_4+d_5$\\
\hline
{\bf 7} & 4 & 45 & $\mathrm{D}_4$ & $d_5-d_6, d_3-d_4, d_2+d_4+d_6, d_1+d_4+d_6$\\
{\bf 8} & 4 & 216 & $\mathrm{A}_4$ & $d_5-d_6, d_3+d_4+d_6, d_2+d_4+d_6, d_1+d_4+d_6$\\
9 & 4 & 540 & $\mathrm{A}_3 \times \mathrm{A}_1$ & $d_3+d_4+d_6, d_2+d_3+d_6, d_1+d_4+d_6, d_2+d_4+d_5$\\
10 & 4 & 120 & $\mathrm{A}_2 \times \mathrm{A}_2$ & $d_4-d_5, d_3+d_4+d_5, d_2+d_4+d_5, d_1+d_5+d_6$\\
11 & 4 & 1080 & $\mathrm{A}_2 \times \mathrm{A}_1^{\times 2}$ & $d_1+d_2+d_5,d_2+d_3+d_6,d_1+d_4+d_6,d_2+d_4+d_5$\\
\hline
{\bf 12} & 5 & 27 & $\mathrm{D}_5$ & $d_5-d_6, d_1-d_4, d_1-d_3, d_1-d_2, d_1+d_5+d_6$\\
{\bf 13} & 5 & 36 & $\mathrm{A}_5$ & $d_5+d_4+d_6, d_4-d_6, d_3-d_5, d_2-d_6, d_1-d_5$\\
14 & 5 & 216 & $\mathrm{A}_4 \times \mathrm{A}_1$ & $d_6, d_4, d_3-d_5, d_2+d_5, d_1$\\
15 & 5 & 360 & $\mathrm{A}_2^{\times 2} \times \mathrm{A}_1$ & $d_2+d_4+d_5,d_2-d_3,d_4-d_5,d_2+d_3+d_6,d_1+d_4+d_6$
\end{tabular}
\caption{The flats of the $\mathrm{E}_6$ reflection arrangement.}
\label{fig:e6flats}
\end{table}
The Bergman fan of ${\rm E}_6$ is the fan over the nested set complex \cite{ARW}, a $4$-dimensional simplicial complex whose vertices are the
$750 = 36\!+\!120\!+\!270\!+\!45\!+\!216\!+\!27\!+\!36$ irreducible~flats.

We define the {\em Yoshida variety} $\,\mathcal{Y}\,$ to be the closure of the image of the rational map 
\begin{equation}
\label{eq:yoshidamap}
\PP^5 \,\, \buildrel{{\rm linear}}\over{\hookrightarrow}\,\, \PP^{35} \,\,
\buildrel{{\rm monomial}} \over{\dashrightarrow} \,\, \PP^{39},
\end{equation}
where the monomial map is defined by the root subsystems of type $\mathrm{A}_2^{\times 3}$. 
Our name for $\mathcal{Y}$ gives credit to Masaaki Yoshida's explicit computations in \cite{yoshida}. (Warning: there is a closely related variety $\mathcal{Y}$ studied in \cite[\S 3.5]{hunt}. This is not the same as our variety.)
Explicitly, as shown in \cite[Proposition 2.4]{CGL}, the map into $\PP^{39}$ is defined by $30$ bracket monomials like
$[125][126][134][234][356][456]$ and $10$ bracket monomials like
$[{\rm conic}][123][456]$. We divide each of these $40$ expressions by
$\prod_{1 \leq i < j \leq 6} (d_i-d_j)$ to get a product of $9$ linear forms.
Thus the rational map (\ref{eq:yoshidamap}) is 
given  by $40$ polynomials of degree $9$ that factor into roots of ${\rm E}_6$.
The {\em tropical Yoshida variety} ${\rm trop}(\mathcal{Y})$ is the image
of the Bergman fan of ${\rm E}_6$ under the linear map $\mathbb{TP}^{35} \rightarrow \mathbb{TP}^{39}$ defined by the corresponding $40 \times 36$-matrix.
The Yoshida variety $\mathcal{Y}$ has  $40$ singular points \cite[Theorem 5.7]{vG}.
Its open part $\mathcal{Y}^{\circ}$ is the moduli space of marked smooth cubic surfaces \cite[Theorem 3.1]{CGL}.
The blow-up of these points is {\em Naruki's cross ratio variety} $Y^6_{\rm lc}$ (following the notation of \cite{HKT}) from \cite{naruki}.
The situation is analogous to Theorem~\ref{thm:IgusaBBS}.

As defined, we consider ${\rm trop}(\mathcal{Y})$ only as a set, but there is a unique coarsest fan structure on this set.
This was shown in  \cite{HKT}. It is the fan over a $3$-dimensional simplicial complex that was described 
by Naruki \cite{naruki}. We call them the {\em Naruki fan} and {\em Naruki complex}, respectively.
The $76 = 36+40$ vertices correspond to the two types
of boundary divisors on $Y^6_{\rm lc}$: the $36$ divisors coming from the hyperplanes of $\mathrm{E}_6$ (type a) and the 
$40$ exceptional divisors of the blow-up (type b). The types of intersections of these divisors is given in \cite[p.23]{naruki} and is listed in Table~\ref{table:naruki}. The divisors of type (a) correspond to root subsystems of type $\mathrm{A}_1$ and the divisors of type (b) correspond to root subsystems of type $\mathrm{A}_2^{\times 3}$. The Naruki complex is the nested set complex on these subsystems. Its face numbers are as follows:

\begin{table}[h]
\centering
\begin{tabular}{l|l}
type &  number\\
\hline
(a) & 36 \\
(b) & 40 \\
\hline
(aa) & 270 \\
(ab) & 360 \\
\hline
(aaa) & 540 \\
(aab) & 1080 \\
\hline
(aaaa) & 135 \\
(aaab) & 1080
\end{tabular}
\caption{The Naruki complex has $76$ vertices, $630$ edges, $1620$ triangles and $1215$ tetrahedra.}
\label{table:naruki}
\end{table}

\begin{theorem} \label{thm:yoshida}
The Yoshida variety $\mathcal{Y}$ is the intersection in $\PP^{39}$
of a $9$-dimensional linear space and a $15$-dimensional toric variety whose
dense torus $\mathbb{G}_m^{15}$ is the intrinsic torus of  $\mathcal{Y}^{\circ}$.
The tropical compactification $\overline{\mathcal{Y}}$ of $\mathcal{Y}^{\circ}$ induced by the Naruki fan is
the cross ratio variety~$Y^6_{\rm lc}$. 
\end{theorem}

The polytope of the toric variety has $2232$ facets.
Its prime ideal is minimally generated by $8922$ binomials, namely
$120$ of degree $3$, 
$810$ of degree $4$, 
$2592$  of degree $5$, 
$2160$ of degree $6$,  and
$3240$ of degree $8$.
These results, which mirror 
parts (b) and (d) in Proposition \ref{prop:toricvarT},
were found using 
{\tt polymake} \cite{GJ} and {\tt gfan} \cite{jensen}.
The prime ideal of $\mathcal{Y}$ is minimally generated by
$30$ of the binomial cubics together with $30$ linear forms.
A natural choice of such linear forms is
described in \cite[\S 3]{yoshida}. It 
comes from $4$-term Pl\"ucker relations such as
$\,[123][456] - [124][356] + [125][346] + [126][345]$.
There are no linear trinomial relations on $\mathcal{Y}$.

The $750$ rays of the Bergman fan map into ${\rm trop}(\mathcal{Y})$ as follows.
Write $m$ for 
the linear map $ \mathbb{TP}^{35} \rightarrow \mathbb{TP}^{39}$
and $F_i$ for the rays representing  family $i$ of irreducible flats of
Table~\ref{fig:e6flats}.~Then:
\begin{align*}
m(F_1) = m(F_8) = m(F_{13})& \text{ has $36$ elements (a),}\\*
m(F_2)& \text{ has $40$ elements (b),}\\*
m(F_4)& \text{ has $270$ elements.}
\end{align*}
All other rays map to $0$ in $ \mathbb{TP}^{39}$.
Each element in $m(F_4)$ is  the sum of two vectors from $m(F_1)$ which form a cone.  
The image of the Bergman fan of $\mathrm{E}_6$ in $\mathbb{TP}^{39}$
is a fan with $346=36{+}40{+}270$ rays that subdivides the Naruki fan.
That fan structure on ${\rm trop}(\mathcal{Y})$ defines a  modification of
 the Naruki variety~$Y^6_{\rm lc}$.

Here is the finite geometry behind (\ref{eq:yoshidamap}). Let $V = \mathbb{F}_2^6$ with coordinates $x_1, \dots, x_6$. There are two conjugacy classes of nondegenerate quadratic forms on $V$. Fix the non-split form
\[
q(x)\,\, =\,\, x_1x_2 + x_3x_4 + x_5^2 + x_5x_6 + x_6^2.
\]
Then the Weyl group $W(\mathrm{E}_6)$ is the subgroup of $\mathrm{GL}_6(\mathbb{F}_2)$ that preserves this form. Using $q(x)$, we define
an orthogonal (in characteristic $2$, 
this also means symplectic) form by
\[
\langle x, y \rangle \,\,=\,\, q(x+y) - q(x) - q(y).
\]
There is a natural bijection between the 36 positive roots of $\mathrm{E}_6$ and the vectors $x \in V$ with $q(x) \ne 0$. There are $120$ planes $W$ such that $q(x) \ne 0$ for all nonzero $x \in W$. These correspond to subsystems of type $\mathrm{A}_2$. The set of $120$ planes breaks up into $40$ triples of pairwise orthogonal planes. These $40$ triples correspond to the subsystems of type $\mathrm{A}_2^{\times 3}$.

\smallskip

We now come to the case $n=7$. The {\em G\"opel variety} $\mathcal{G}$ of
 \cite{RSSS} is the closed image of a map
 \begin{equation}
\label{eq:gopelmap}
\PP^6 \,\, \buildrel{{\rm linear}}\over{\hookrightarrow}\,\, \PP^{62} \,\,
\buildrel{{\rm monomial}} \over{\dashrightarrow} \,\, \PP^{134}.
\end{equation}
The linear map is given by the $63$ hyperplanes in
the reflection arrangement $\mathrm{E}_7$, and the 
monomial map by the $135$ root subsystems of type $\mathrm{A}_1^{\times 7}$.
The full list of all flats of the arrangement $\mathrm{E}_7$ appears in \cite[Table 2]{RSSS}.
In \cite[Corollary 9.2]{RSSS} we argued that the
{\em tropical G\"opel variety} ${\rm trop}(\mathcal{G})$ is the image of
the Bergman fan of $\mathrm{E}_7$ under the induced linear map
$\mathbb{TP}^{62} \rightarrow \mathbb{TP}^{134}$, 
and we asked how ${\rm trop}(\mathcal{G})$ would be related to the fan for $Y^7_{\rm lc}$ in~\cite[\S 1.14]{HKT}.
We call that fan the {\em Sekiguchi fan}, after \cite{sekiguchi}.
The following theorem answers our question.

\begin{theorem} \label{thm:goepel}
The G\"opel variety $\mathcal{G}$ is the intersection in $\PP^{134}$ of a $14$-dimensional linear space and a $35$-dimensional toric variety whose
dense torus $\mathbb{G}_m^{35}$ is the intrinsic torus of $\mathcal{G}^{\circ}$.
The tropical compactification $\overline{\mathcal{G}}$ 
of  the open G\"opel variety $\mathcal{G}^{\circ}$ induced by the Sekiguchi fan is the Sekiguchi variety $Y^7_{\rm lc}$. 
Hence, the Sekiguchi fan is the coarsest fan structure on ${\rm trop}(\mathcal{G})$.
\end{theorem}

The result about  the linear space and the toric variety is \cite[Theorem 6.2]{RSSS}.
The determination of the intrinsic tori in Theorems \ref{thm:yoshida} and \ref{thm:goepel} is immediate from Lemma \ref{lem:intrinsictorusimage}. The last assertion follows from 
the fact that the open G\"opel variety $\mathcal{G}^{\circ}$ is the moduli space of marked smooth del Pezzo surfaces of degree two.
For this see \cite[Theorem 3.1]{CGL}.

The Bergman fan of type $\mathrm{E}_7$ has $6091$ rays. They are listed in \cite[Table 2]{RSSS}.
The $6091$ rays  map into $\mathrm{trop}(\mathcal{G})$ as follows.
Write $F_i$ for family $i$ in \cite[Table 2]{RSSS}. Then:
\begin{align*}
m(F_1) = m(F_{17}) = m(F_{25}) \text{ has $63$ elements,}\\*
m(F_2) = m(F_{15}) \text{ has $336$ elements,}\\*
m(F_4) \text{ has $630$ elements,}\\*
m(F_{24}) \text{ has $36$ elements,}\\*
m(F_8) \text{ has $2016$ elements,}\\*
m(F_9) \text{ has $315$ elements,}\\*
m(F_{16}) \text{ has $1008$ elements.}
\end{align*}
Finally, $m$ sends $F_{26}$ to 0 (multiple of all $1$'s vector). The fan on the first $4$ types of rays is the Sekiguchi fan as described in \cite[\S 1.14]{HKT}. The image of the Bergman fan of ${\rm E}_7$ is a refinement of the Sekiguchi fan, as follows:
\begin{compactitem}
\item Every ray in $m(F_8)$ is uniquely the sum of a ray in $m(F_2)$ and a ray in $m(F_{24})$. This is in the image of a cone of nested set type $\mathrm{A}_2 \subset \mathrm{A}_6$.
\item Every ray in $m(F_9)$ is uniquely the sum of three rays in $m(F_1)$. This is in the image of a cone of nested set type $\mathrm{A}_1^{\times 3}$.
\item Every ray in $m(F_{16})$ can be written uniquely as a positive sum of a ray in $m(F_1)$ and a ray in $m(F_{24})$. This is in the image of a cone of nested set type $\mathrm{A}_1 \subset \mathrm{A}_6$.
\end{compactitem}

\smallskip

The Sekiguchi fan on ${\rm trop}(\mathcal{G})$ is a fan over a $5$-dimensional simplicial complex with $1065 = 63+336+630+36$ vertices. 
It has $9$ types of facets, corresponding to the $9$ tubings
shown in \cite[Figure 2, page 200]{HKT}.
The significance of the Naruki fan and the Sekiguchi fan lies in the commutative diagram in \cite[Lemma 5.4]{HKT}, which we restate here:
\begin{equation}
\begin{diagram}
\mathbb{P}^6 & \rDashto  & \mathcal{G}^{\circ} \\
\dDashto                       &        & \dTo \\
\mathbb{P}^5    & \rDashto & \mathcal{Y}^{\circ}
\end{diagram}
\end{equation}
The horizontal maps are those in (\ref{eq:gopelmap}) and (\ref{eq:yoshidamap}). The left vertical map is defined by dropping a coordinate. The tropicalization of the right vertical map $\mathcal{G}^{\circ} \rightarrow \mathcal{Y}^{\circ}$ is a linear projection
\begin{equation}
\label{eq:GtoT}
{\rm trop}(\mathcal{G}) \,\,\rightarrow \,\,{\rm trop}(\mathcal{Y})
\end{equation}
from the tropical G\"opel variety onto the tropical Yoshida variety.

 We wish to explicitly determine this map on each cone of  ${\rm trop}(\mathcal{G})$.
 The point is that all tropicalized generic del Pezzo surfaces of degree $3$ appear in the fibers of (\ref{eq:GtoT}),
by the result about the universal family in \cite[Theorem 1.2]{HKT}, 
and our Theorems \ref{thm:yoshida} and \ref{thm:goepel}. At infinity, such a 
del Pezzo surface
is glued from $27$ trees, which are exactly the tropical image of the $27$ lines on a cubic surface over $K$. Each tree has $10$ leaves, which come from the intersections of the $27$ lines. Thus, each tree represents a point of $\mathcal{M}_{0,10}(K)$.
Thus tropicalized del Pezzo surfaces of degree $3$ can be represented by
 a {\em tree arrangement} in the sense of \cite[\S 4]{HJJS}. 
 
 One issue with the map \eqref{eq:GtoT} is that its
 zero fiber is $3$-dimensional. Namely, it the union of tropicalizations of 
 all constant coefficient cubic surfaces. The zero fiber has $27$ rays, one for each line on the cubic surface, and $45$ triangular cones, one for each triple of pairwise intersecting lines. This is the subtle issue of 
 {\em Eckhart points}, addressed by \cite[Theorem 1.19]{HKT}. Cubic surfaces with Eckhart points are special, for they contribute to the points in the interior of the $45$ triangular cones. Disallowing these removes
the interiors of the triangular cones, and we are left with a balanced two-dimensional fan.
This is the fan over a graph with $27$ vertices and $135$ edges,
representing generic constant coefficient cubic~surfaces.

\smallskip

In this section, we developed some tools for the classification
of tropical cubic surfaces, namely as fibers of \eqref{eq:GtoT},
but we did not actually carry out this classification.
That problem will be solved in a forthcoming paper by
  Qingchun Ren, Kristin Shaw and Bernd Sturmfels.

\noindent
\footnotesize {\bf Authors' address}:
Department of Mathematics, University of California, Berkeley, CA 94720, USA. \hfill
{\tt \{qingchun,svs,bernd\}@math.berkeley.edu}
\end{document}